\documentclass[final,12pt]{amsart}

\usepackage{pifont,geometry}

\usepackage{amssymb,amsmath,amsfonts,amsthm,amscd,stmaryrd}

\theoremstyle{plain}
\newtheorem{thm}{Theorem}
\newtheorem{cor}[thm]{Corollary}
\newtheorem{lem}[thm]{Lemma}
\newtheorem{prop}[thm]{Proposition}
\theoremstyle{definition}
\newtheorem{definition}[thm]{Definition}
\newtheorem{rem}[thm]{Remark}
\newtheorem{exam}[thm]{Example}

\numberwithin{thm}{section}
\numberwithin{equation}{section}

\newcommand{\EQ}[1]{\eqref{eq:#1}} 
\newcommand{\LEM}[1]{Lemma~\ref{lem:#1}}    
\newcommand{\DEF}[1]{Definition~\ref{def:#1}}    
\newcommand{\THM}[1]{Theorem~\ref{thm:#1}}  
\newcommand{\REM}[1]{Remark~\ref{rem:#1}}  
\newcommand{\PROP}[1]{Proposition~\ref{prop:#1}}  
\newcommand{\COR}[1]{Corollary~\ref{cor:#1}} 
\newcommand{\SEC}[1]{Section~\ref{sec:#1}}  

\newcommand{\APP}[1]{Appendix~\ref{app:#1}}

\newcommand{\ENUM}[1]{\ref{enum:#1}}
\newcounter{hypo}

\newcommand{\HYP}[1]{\ref{hyp:#1}}
\newcommand{\maxboth}{\max\left\{\lambda^-_1(F,\Omega),\lambda^+_1(F,\Omega)\right\}}
\newcommand{\minboth}{\min\left\{\lambda^-_1(F,\Omega),\lambda^+_1(F,\Omega)\right\}}
\newcommand{\consp}{\delta_1}
\newcommand{\consz}{\delta_0}

\DeclareMathOperator{\trace}{trace}

\newcommand{\puccisub}[2]{\mathcal{P}^-_{#1,#2}}
\newcommand{\Puccisub}[2]{\mathcal{P}^+_{#1,#2}}
\newcommand{\pucci}{\mathcal{P}^-}
\newcommand{\Pucci}{\mathcal{P}^+}
\newcommand{\Fup}{F^*}
\newcommand{\Fdown}{F_*}
\newcommand{\inOmega}{\quad \mbox{in} \quad \Omega}
\newcommand{\onOmega}{\quad \mbox{on} \quad \partial \Omega}
\newcommand{\R}{\ensuremath{\mathbb{R}}}
\newcommand{\Z}{\ensuremath{\mathbb{Z}}}

\newcommand{\Sy}{\ensuremath{\mathbb{S}^n}}
\renewcommand{\labelenumi}{(\roman{enumi})}

\begin{document}
\title[Principal eigenvalues and an anti-maximum principle]{Principal eigenvalues and an anti-maximum principle for homogeneous fully nonlinear elliptic equations}
\author{Scott N. Armstrong}
\address{Department of Mathematics, University of California, Berkeley, CA 94720.}
\email{sarm@math.berkeley.edu}
\date{October 24, 2008}

\begin{abstract}
We study the fully nonlinear elliptic equation
\begin{equation*}
F(D^2u,Du,u,x) = f
\end{equation*}
in a smooth bounded domain $\Omega$, under the assumption that the nonlinearity $F$ is uniformly elliptic and positively homogeneous. Recently, it has been shown that such operators have two principal ``half" eigenvalues, and that the corresponding Dirichlet problem possesses solutions, if both of the principal eigenvalues are positive. In this paper, we prove the existence of solutions of the Dirichlet problem if both principal eigenvalues are negative, provided the ``second" eigenvalue is positive, and generalize the anti-maximum principle of Cl\'{e}ment and Peletier \cite{Clement:1979} to homogeneous, fully nonlinear operators.
\end{abstract}

\maketitle



\section{Introduction}

This paper is a contribution to the study of viscosity solutions of the uniformly elliptic, fully nonlinear partial differential equation
\begin{equation}\label{eq:maineq}
F(D^2u,Du,u,x) = f
\end{equation}
in a bounded domain $\Omega\subseteq \R^n$, subject to the Dirichlet boundary condition
\begin{equation}\label{eq:mainbc}
u = 0 \onOmega.
\end{equation}
The problem \EQ{maineq}-\EQ{mainbc} possesses a unique solution under the assumption the nonlinearity $F$ is \emph{proper}; that is, the map
\begin{equation*}
z \mapsto F(M,p,z,x) \quad \mbox{is nondecreasing}.
\end{equation*}
On the other hand, the Fredholm theory of compact linear operators provides a complete understanding of the existence and uniqueness of solutions of \EQ{maineq}-\EQ{mainbc} in the case $F=L$ is linear (see, for example, \cite{Evans:Book}). In particular, in this case the Dirichlet problem \EQ{maineq}-\EQ{mainbc} has a unique solution if and only if $0$ is not an eigenvalue of $F$.

\medskip

Recently, there has been much interest in studying \EQ{maineq}-\EQ{mainbc} for nonlinear operators that are not necessarily proper. Quaas and Sirakov~\cite{Quaas:2006,Quaas:2008} have shown that a nonlinear operator $F$ which is uniformly elliptic, as well as positively homogeneous and convex (or concave) in $(D^2u,Du,u)$, possesses two principal ``half" eigenvalues $\lambda^+_1(F,\Omega)$ and $\lambda^-_1(F,\Omega)$. The former corresponds to a positive eigenfunction $\varphi^+_1 > 0$ and the latter to a negative eigenfunction $\varphi^-_1 < 0$. In \cite{Quaas:2008} it was also demonstrated that the Dirichlet problem \EQ{maineq}-\EQ{mainbc} has a unique viscosity solution provided that both of these principal eigenvalues are positive.

Ishii and Yoshimura \cite{Ishii:preprint} have independently proven analogous results for operators which are not necessarily convex, such as the Bellman-Isaacs operator, and Birindelli and Demengel \cite{Birindelli:2006,Birindelli:2007} have shown similar results for certain nonlinear operators which are degenerate elliptic. All of these papers owe much to the work of Lions \cite{Lions:1983d}, who used stochastic methods to study the principal half-eigenvalues of certain Bellman operators, and also the ideas of Berestycki, Nirenberg and Varadhan \cite{Berestycki:1994}, who discovered deep connections between the maximum principle and principal eigenvalues of linear operators.

\medskip

In this paper, we provide new proofs of the results mentioned above. We then show that for positively homogeneous, fully nonlinear operators $F$, the Dirichlet problem \EQ{maineq}-\EQ{mainbc} possesses solutions provided
\begin{equation}\label{eq:existence-range}
\max \{ \lambda^+_1(F,\Omega), \lambda^-_1(F,\Omega) \} < 0 < \lambda_2(F,\Omega),
\end{equation}
where $\lambda_2(F,\Omega)$ is the infimum of all eigenvalues of $F$ which are larger than both $\lambda^+_1(F,\Omega)$ and $\lambda^-_1(F,\Omega)$. This result is new even for convex $F$, and so far as we know, is the first general existence result shown for a wide class of fully nonlinear operators which do not satisfy a comparison principle. Finally, studying certain of these solutions, we generalize the anti-maximum principle of Cl\'{e}ment and Peletier \cite{Clement:1979} to fully nonlinear equations.

\medskip

The phenomena of nonlinear operators possessing two principal half-eigenvalues was first noticed long ago by Pucci \cite{Pucci:1966}, who found explicit formulas for the eigenvalues and eigenfunctions of a specific operator (similar to the Pucci maximal and minimal operators) on a ball. It was also discovered by Berestycki \cite{Berestycki:1977} for Sturm-Liouville equations. For more on principal eigenvalues of nonlinear elliptic operators, we refer to \cite{Busca:1999,Busca:2005,Felmer:2004,Juutinen:preprint,Patrizi:preprint1,Patrizi:preprint2}.

Much of the recent work on eigenvalues of nonlinear operators is based on a deep connection with the maximum principle, exploited in the linear case by Berestycki, Nirenberg and Varadhan \cite{Berestycki:1994}. If $L$ is a linear operator then the maximum principle holds for the operator $L-\mu$ in $\Omega$ for $\mu$ in some open interval $(-\infty, \rho)$. In fact, $\lambda_1(L,\Omega) = \rho$. Alternatively, $\lambda_1(L,\Omega)$ can be characterized as the supremum of all $\mu$ for which there exists a positive supersolution $u > 0$ of the equation
\begin{equation*}
L u - \mu u \geq 0 \inOmega.
\end{equation*}
In \cite{Berestycki:1994}, these facts are generalized using arguments which rely not on linearity, but on homogeneity. It is these powerful techniques which open up the study of \EQ{maineq}-\EQ{mainbc} for nonlinear operators, as we will see below.

\medskip

Precise statements of our results are contained in \SEC{statement-of-results}. In \SEC{principal-eigenvalues}, we present a proof of the existence and investigate the basic properties of the principal eigenvalues $\lambda_1^+(F,\Omega)$ and $\lambda^-_1(F,\Omega)$. We study existence and uniqueness of solutions of the Dirichlet problem in \SEC{Dirichlet-problem}, in the case that at least one of the principal eigenvalues is positive. Our main results are the content of Sections \ref{sec:existence-past-lambda-1} and \ref{sec:anti-maximumprinciple}. First, in \SEC{existence-past-lambda-1}, we use the theory of Leray-Schauder degree to obtain the existence of viscosity solutions of the Dirichlet problem \EQ{maineq}-\EQ{mainbc} under the assumption that \EQ{existence-range} holds. In \SEC{anti-maximumprinciple}, we show that certain of these solutions satisfy an anti-maximum principle. For the reader's convenience, we have also included a proof of Hopf's Lemma for viscosity solutions in \APP{hopf}.

\medskip

This paper was completed while I was a Ph.D. student at the University of California, Berkeley. I would like to thank the Department of Mathematics for its support, and to express my gratitude to my thesis advisor, Prof. Lawrence C. Evans, for many years of patient guidance and encouragement. I also gratefully acknowledge the helpful comments and references I received from Prof. Boyan Sirakov, Prof. Hitoshi Ishii, and the referee.



\section{Statements of main results}\label{sec:statement-of-results}

Throughout this paper, we take $\Omega$ to be a bounded, smooth, and connected open subset of $\R^n$. We denote the set of $n$-by-$n$ symmetric matrices by $\Sy$. For $M \in \Sy$ and $0 < \gamma \leq \Gamma$, define
\begin{equation*}
\Puccisub{\gamma}{\Gamma} (M) = \sup_{A\in \llbracket\gamma,\Gamma\rrbracket} \left[ - \trace(AM) \right] \quad \mbox{and} \quad \puccisub{\gamma}{\Gamma} (M) = \inf_{A\in \llbracket\gamma,\Gamma\rrbracket} \left[ - \trace(AM) \right],
\end{equation*}
where the set $\llbracket\gamma,\Gamma\rrbracket \subseteq \Sy$ consists of the symmetric matrices the eigenvalues of which lie in the interval $[ \gamma, \Gamma]$. The nonlinear operators $\Puccisub{\gamma}{\Gamma}$ and $\puccisub{\gamma}{\Gamma}$ are the Pucci extremal operators. To ease our notation, we will often drop the subscripts and write $\Pucci$ and $\pucci$. See \cite{Caffarelli:Book,Caffarelli:1996} for basic properties of the Pucci operators.

\medskip

We require that our nonlinear operator
\begin{equation*}
F:\Sy \times \R^n \times \R \times \Omega \rightarrow \R
\end{equation*}
satisfies the following hypotheses:
\renewcommand{\labelenumi}{(F\arabic{hypo})}
\begin{enumerate}
\usecounter{hypo}
\item For each $K > 0$, there exist an increasing continuous function $\omega_K:[0,\infty)\to [0,\infty)$ for which $\omega_K(0)=0$, and a positive constant $\frac{1}{2}< \nu \leq 1$, depending on $K$, such that
\begin{equation*}
|F(M,p,z,x) - F(M,p,z,y) | \leq \omega_K\left( |x-y|^\nu (|M|+1) \right)
\end{equation*}
for all $M\in \Sy$, $p\in \R^n$, $z\in \R$, and $x,y\in \Omega$ satisfying $|p|, |z| \leq K$. \label{hyp:Fcontinuous}
\item There exist constants $\consp,\consz\geq 0$ and $0<\gamma \leq \Gamma$ such that
\begin{multline*}
\qquad \quad \puccisub{\gamma}{\Gamma}(M-N) - \consp |p-q| - \consz |z-w| \leq F(M,p,z,x) - F(N,q,w,x) \\ 
\leq \Puccisub{\gamma}{\Gamma}(M-N) + \consp |p-q| + \consz|z-w|
\end{multline*}
for all $M,N\in \Sy$, $p,q\in R^n$, $z,w \in \R$, $x\in \Omega$.\label{hyp:Felliptic}
\item $F$ is positively homogeneous of order one, jointly in its first three arguments; i.e.,
\begin{equation*}
F(tM,tp,tz,x) = tF(M,p,z,x)\quad \mbox{for all} \quad t\geq 0
\end{equation*}
and all $M\in \Sy$, $p\in \R^n$, $z\in \R$, $x\in \Omega$.\label{hyp:Fhomogeneous}
\end{enumerate}
\renewcommand{\labelenumi}{(\roman{enumi})}
Hypothesis \HYP{Felliptic} implies $F$ is uniformly elliptic in the familiar sense that
\begin{equation*}
- \gamma \trace (N) \geq F(M+N,p,z,x) - F(M,p,z,x).
\end{equation*}
for every nonnegative definite matrix $N\in \Sy$. (In particular, we adopt the convention that $-\Delta$ and not $\Delta$ is an elliptic operator.)

\medskip

Consider a family $\{ L^{\alpha,\beta}\}$ of linear, uniformly elliptic operators
\begin{equation*}
L^{\alpha,\beta} u = - a^{ij}_{\alpha,\beta} u_{ij} + b^j_{\alpha,\beta} u_j + c_{\alpha,\beta} u
\end{equation*}
indexed by parameters $\alpha \in A$ and $\beta \in B$. The nonlinear operator defined by
\begin{equation}\label{eq:BIoperator}
F(D^2u,Du,u,x) = \inf_{\alpha \in A} \sup_{\beta\in B} \ L^{\alpha,\beta} u
\end{equation}
is called the \emph{Bellman-Isaacs} operator, which occurs naturally in the theory of stochastic differential games (see \cite{Fleming:1989,Friedman:Book2006}). The operator $F$ given by \EQ{BIoperator} satisfies our hypotheses \HYP{Fcontinuous}, \HYP{Felliptic}, and \HYP{Fhomogeneous}, provided the coefficients $a^{ij}_{\alpha,\beta}$ are bounded in $C^{\nu}(\Omega)$ as well as uniformly elliptic, uniformly in the parameters, and the families $\{ b^j_{\alpha,\beta} \}$ and $\{ c_{\alpha,\beta} \}$ are uniformly bounded and equicontinuous. 

\medskip

All differential equations and inequalities appearing in this paper are assumed to be satisfied in the viscosity sense. We now briefly recall the notion of viscosity solutions.
\begin{definition}
Assume $f\in C(\Omega)$. We say that $u\in C(\Omega)$ is a \emph{viscosity subsolution (supersolution)} of the equation
\begin{equation}\label{eq:defviscsol}
F(D^2u,Du,u,x) = f \inOmega
\end{equation}
if, for every $x_0\in \Omega$ and every $\varphi\in C^2(\Omega)$ for which
\begin{equation*} 
x \mapsto u(x) - \varphi(x) \quad \mbox{has a local maximum (minimum) at} \quad x_0,
\end{equation*}
we have
\begin{equation*}
F(D^2\varphi(x_0),D\varphi(x_0),u(x_0),x_0) \leq \ (\geq)\ f(x_0).
\end{equation*}
We say that $u$ is a \emph{viscosity solution} of \EQ{defviscsol} if it is both a viscosity subsolution and supersolution of \EQ{defviscsol}.
\end{definition}
See \cite{Caffarelli:Book,UsersGuide} for an introduction to the notion of viscosity solutions. In contrast to the papers \cite{Birindelli:2006,Quaas:2008}, we do not allow our nonlinear operator $F$ to be only measurable in $x$. However, many of our results can be generalized to this case, provided that $F$ is convex (or concave) in its first argument. This assumption allows us to make use of $W^{2,p}$ estimates for $L^p$-viscosity solutions (see Winter \cite{Winter:2008}, who has recently extended the interior estimates of Caffarelli \cite{Caffarelli:1989} to the boundary), which we need in order to dispense with certain technicalities arising in the measurable case. See \cite{Caffarelli:1996,Crandall:1999,Crandall:2000,Quaas:2008,Winter:2008} for details.

\medskip

It is well known (for example, see \cite{Donsker:1975}) that the principal eigenvalue $\lambda_1(L,\Omega)$ of a linear elliptic operator $L$ in $\Omega$ can be expressed by the max-min formula
\begin{equation*}
\lambda_1(L,\Omega) = \sup_{\varphi > 0} \inf_{x\in \Omega} \frac{(L\varphi)(x)}{\varphi(x)},
\end{equation*}
where the supremum is taken over all positive functions $\varphi \in C^2(\bar{\Omega})$. With this in mind, and following \cite{Berestycki:1994} and \cite{Quaas:2008}, we define the constants
\begin{equation}
\lambda^+_1(F,\Omega) =  \sup \left\{  \rho \, : \, \exists v\in C(\Omega), v > 0 \mbox{ and } F(D^2v,Dv,v,x) \geq \rho v \mbox{ in } \Omega \right\},
\end{equation}
and
\begin{equation}
\lambda^-_1(F,\Omega) =  \sup \left\{ \rho \, : \, \exists v\in C(\Omega), v < 0 \mbox{ and } F(D^2v,Dv,v,x) \leq \rho v \mbox{ in } \Omega \right\}.
\end{equation}
Then $\lambda^+_1(F,\Omega)$ and $\lambda^-_1(F,\Omega)$ are the principal half-eigenvalues of $F$ in $\Omega$:

\begin{thm}[Ishii and Yoshimura \cite{Ishii:preprint}]\label{thm:eigenvalues}
There exist functions $\varphi^+_1,\varphi^-_1 \in C^{1,\alpha}(\Omega)$ such that $\varphi^+_1 > 0$ and $\varphi^-_1<0$ in $\Omega$, and which satisfy
\begin{equation}\label{eq:Feigenvalues}
\left\{ \begin{aligned}
F(D^2\varphi_1^+,D\varphi_1^+,\varphi_1^+,x) = {}& \lambda_1^+(F,\Omega) \varphi_1^+ & \mbox{in} & \quad \Omega \\
F(D^2\varphi_1^-,D\varphi_1^-,\varphi_1^-,x) = {}& \lambda_1^-(F,\Omega) \varphi_1^- & \mbox{in} & \quad \Omega \\
\varphi_1^+ = \varphi^-_1 = {}& 0 & \mbox{on} & \quad \partial \Omega. 
\end{aligned}\right.
\end{equation}
Moreover, the eigenvalue $\lambda_1^+(F,\Omega)$ ($\lambda_1^-(F,\Omega)$) is unique in the sense that if $\rho$ is another eigenvalue of $F$ in $\Omega$ associated with a nonnegative (nonpositive) eigenfunction, then $\rho=\lambda_1^+(F,\Omega)$ ($\rho=\lambda_1^-(F,\Omega)$); and is simple in the sense that if $\varphi \in C(\bar{\Omega})$ is a solution of \EQ{Feigenvalues} with $\varphi$ in place of $\varphi^+_1$ ($\varphi^-_1$), then $\varphi$ is a constant multiple of $\varphi_1^+$ ($\varphi_1^-$).
\end{thm}

We will present a simpler proof of \THM{eigenvalues} than appears in \cite{Ishii:preprint}, using many of the ideas in the recent papers \cite{Birindelli:2007,Quaas:2008}, and of course the methods in \cite{Berestycki:1994}. Using \THM{eigenvalues} it is simple to show using the Perron method, that the Dirichlet problem
\begin{equation}\label{eq:Dirichlet-problem}
\left\{ \begin{aligned}
F(D^2u,Du,u,x) = {}& \lambda u + f & \mbox{in} & \quad \Omega \\
u = {}& 0 & \mbox{on} & \quad \partial \Omega.
\end{aligned}\right.
\end{equation}
possesses solutions provided that the principal eigenvalues are positive, and for $f$ with a certain sign if only one principal eigenvalue is positive.

\begin{thm}[Ishii and Yoshimura \cite{Ishii:preprint}]\label{thm:DP-existence}
Assume $f\in C(\Omega) \cap L^p(\Omega)$ for some $p > n$.
\begin{enumerate}
\item If $f\geq 0$ and $\lambda < \lambda_1^+(F,\Omega)$, then the Dirichlet problem \EQ{Dirichlet-problem} has a unique nonnegative solution $u \in C(\bar{\Omega})$. Moreover, $u \in C^{1,\alpha}(\Omega)$. \label{enum:DP-existence-plus}
\item If $f \leq 0$ and $\lambda < \lambda_1^-(F,\Omega)$, then the Dirichlet problem \EQ{Dirichlet-problem} has a unique nonpositive solution $u\in C^{1,\alpha}(\Omega)$.
\label{enum:DP-existence-minus}
\item If $\lambda < \minboth$, then the Dirichlet problem \EQ{Dirichlet-problem} has a solution $u \in C^{1,\alpha}(\Omega)$. \label{enum:DP-existence-both}
\end{enumerate}
\end{thm}

If $F$ is convex or concave, then the Dirichlet problem \EQ{Dirichlet-problem} in fact has \emph{unique} solutions, as was shown in \cite{Quaas:2008}, in the case that both principal eigenvalues are positive. For general homogeneous $F$, we do not know if the solutions obtained in \ENUM{DP-existence-both} are unique. However, in \SEC{Dirichlet-problem} we give a sufficient condition, discovered in \cite{Ishii:preprint}, from which we also recover uniqueness for convex operators.

\medskip

In contrast to the situation for linear operators, while no $\lambda$ satisfying
\begin{equation}\label{eq:lambda-in-the-middle}
\minboth < \lambda < \maxboth
\end{equation}
is an eigenvalue of $F$ in $\Omega$, neither for such $\lambda$ do we have general existence or uniqueness of solutions of \EQ{Dirichlet-problem}. See \SEC{anti-maximumprinciple} for some nonexistence results, and Sirakov \cite{Sirakov:preprint} for much more on the failure of existence and uniqueness of solutions of \EQ{Dirichlet-problem}, for $\lambda$ between the two half-eigenvalues.

\medskip

Our main results concern the existence and behavior of solutions of the Dirichlet problem \EQ{Dirichlet-problem} for
\begin{equation*}
\lambda > \maxboth.
\end{equation*}
Define
\begin{equation}\label{eq:def-lambda-2}
\lambda_2(F,\Omega) = \inf \left\{ \rho > \maxboth \, : \, \rho \mbox{ is an eigenvalue of } F \mbox{ in } \Omega \right\}.
\end{equation}
Notice the possibility that $\lambda_2(F,\Omega) = +\infty$. This can occur, for example, if $F$ is a linear operator which is not symmetric. However, as we will show in \LEM{lambda-2-is-eigenvalue}, $\lambda_2(F,\Omega)$ is an eigenvalue of $F$ in $\Omega$ provided it is finite. Moreover, $\lambda_2(F,\Omega) > \maxboth$. Employing the theory of Leray-Schauder degree, we will prove the following theorem. 

\begin{thm}\label{thm:existence-past-lambda-1}
Assume $f\in C(\Omega) \cap L^p(\Omega)$ for some $p > n$, and
\begin{equation*}
\maxboth < \lambda < \lambda_2(F,\Omega).
\end{equation*}
Then there exists a solution $u\in C^{1,\alpha}(\Omega)$ of the Dirichlet problem \EQ{Dirichlet-problem}.
\end{thm}

The proof of \THM{existence-past-lambda-1} is similar to degree-theoretic arguments employed in other contexts. See, for example, the papers \cite{Busca:2005,delPino:1989,delPino:1991}.

\medskip

Our last result concerns the behavior of certain of the solutions of \EQ{Dirichlet-problem}, which exist according to \THM{existence-past-lambda-1}. It is well known (see, e.g., \cite{UsersGuide}) that if $F$ is proper, then $F$ satisfies the comparison principle in any domain. In particular, if $F$ is proper and $u\in C(\bar{\Omega})$ is a solution of the problem
\begin{equation*}
\left\{ \begin{aligned}
F(D^2u,Du,u,x) \geq {}& 0 & \mbox{in} & \quad \Omega \\
u \geq {}& 0 & \mbox{on} & \quad \partial \Omega,
\end{aligned}\right.
\end{equation*}
then $u \geq 0$ in $\Omega$. If in addition $u \not\equiv 0$, then $u > 0$ in $\Omega$ by Hopf's Lemma. We will see later that the same conclusion holds if we replace the assumption that $F$ is proper with the less restrictive condition $\lambda^-_1(F,\Omega) > 0$. In contrast, we will demonstrate that the \emph{opposite} conclusion holds if $\lambda^+_1(F,\Omega) \leq \lambda^-_1(F,\Omega) = -\alpha < 0$, provided $\alpha > 0$ is sufficiently small: the function $u < 0$ in $\Omega$.

\begin{thm}\label{thm:AMP}
Let $p > n$, and suppose $f\in C(\Omega)\cap L^p(\Omega)$ is such that $f\not\equiv 0$.
\begin{enumerate}
\item \label{enum:AMP1} If $f \geq 0$, then there exists a small positive constant $\eta = \eta(f) >0$, such that if
\begin{equation*}
\lambda^+_1(F,\Omega) \leq \lambda_1^-(F,\Omega) < \lambda \leq \lambda_1^-(F,\Omega) + \eta,
\end{equation*}
then any solution $u\in C(\bar{\Omega})$ of \EQ{Dirichlet-problem} satisfies $u < 0$ in $\Omega$.
\item \label{enum:AMP2} If $f \leq 0$, then there exists a small positive constant $\eta=\eta(f) > 0$, such that if
\begin{equation*}
\lambda^-_1(F,\Omega) \leq \lambda^+_1(F,\Omega) < \lambda \leq \lambda^+_1(F,\Omega) + \eta,
\end{equation*}
then any solution $u\in C(\bar{\Omega})$ of \EQ{Dirichlet-problem} satisfies $u > 0$ in $\Omega$.
\end{enumerate}
\end{thm}
A generalization of the well-known anti-maximum principle discovered by Cl\'{e}ment and Peletier~\cite{Clement:1979} in the linear case, \THM{AMP} is, to our knowledge, the first result of its kind for a wide class of fully nonlinear operators. (However, see Godoy, Gossez and Paczka \cite{Godoy:2002} for an anti-maximum principle for the $p$-Laplacian operator.) The proof of \THM{AMP} is based on an indirect argument due to Birindelli \cite{Birindelli:1995}.
 
\medskip

Our analysis in this paper makes use of $C^{1,\alpha}$ estimates for viscosity solutions of fully nonlinear equations (see Trudinger \cite{Trudinger:1988}, as well as \cite{Winter:2008}). However, if $F$ is convex (or concave) in $M$, and is sufficiently regular in $x$, then we may instead use the Evans-Krylov $C^{2,\alpha}$ estimates (see, for example, \cite{Caffarelli:Book, Gilbarg:Book}) to deduce that any solution $u\in C(\bar{\Omega})$ of the Dirichlet problem \EQ{Dirichlet-problem} is actually a classical solution $u\in C^{2,\alpha}(\Omega)$ provided that $f\in C^\alpha(\Omega)$. For example, the nonlinear operator 
\begin{equation*}
G(D^2u,Du,u,x) = \Pucci(D^2u) + g(Du,u,x)
\end{equation*}
has principal eigenfunctions belonging to the space $C^{2,\alpha}(\Omega)$, provided the function $g=g(p,z,x)$ is positively homogeneous in $(p,z)$ and Lipschitz in $(p,z,x)$. See Section 17.5 of \cite{Gilbarg:Book} for more details.



\section{Principal eigenvalues and the maximum principle} \label{sec:principal-eigenvalues}

In this section, we will explore the relationship between the maximum principle and positive viscosity supersolutions (and negative viscosity subsolutions) of the equation
\begin{equation*}
F(D^2u,Du,u,x) = 0 \inOmega. 
\end{equation*}
We will then prove \THM{eigenvalues}, as well as a few useful facts regarding $\lambda^+_1(F,\Omega)$ and $\lambda^-_1(F,\Omega)$.

\begin{definition}
We say the nonlinear operator $F$ \emph{satisfies the maximum principle in $\Omega$} if, whenever $v\in C(\bar{\Omega})$ is a subsolution of
\begin{equation*}
\left\{ \begin{aligned}
F(D^2v,Dv,v,x) \leq {}& 0 & \mbox{in} & \quad \Omega \\
v \leq {}& 0 & \mbox{on} & \quad \partial \Omega,
\end{aligned}\right.
\end{equation*}
we have $v\leq 0$ in $\Omega$. Similarly, we say $F$ \emph{satisfies the minimum principle in $\Omega$} if, for any $v\in C(\bar{\Omega})$ satisfying
\begin{equation*}
\left\{ \begin{aligned}
F(D^2v,Dv,v,x) \geq {}& 0 & \mbox{in} & \quad \Omega \\
v \geq {}& 0 & \mbox{on} & \quad \partial \Omega,
\end{aligned}\right.
\end{equation*}
we have $v\geq 0$ in $\Omega$. Finally, we say $F$ satisfies the \emph{comparison principle in $\Omega$} if, whenever $f\in C(\Omega)$ and $u,v \in C(\bar{\Omega})$ satisfy
\begin{equation}\label{eq:def-CP-1}
F(D^2u,Du,u,x) \leq f \leq F(D^2v,Dv,v,x) \inOmega
\end{equation}
as well as
\begin{equation}\label{eq:def-CP-2}
u \leq v \onOmega,
\end{equation}
we have $u \leq v$ in $\Omega$.
\end{definition}

As we mentioned earlier, the nonlinear operator $F$ satisfies the comparison principle in any domain provided it is proper (see \cite{UsersGuide}). We will require the following technical lemma.

\begin{lem}\label{lem:u-v-subsolution}
Assume $F$, $G$ and $H$ are nonlinearities which satisfy \HYP{Fcontinuous} and \HYP{Felliptic}, and
\begin{equation}\label{eq:H-geq-F-plus-G}
H(M+N,p+q,z+w,x) \leq F(M,p,z,x) + G(N,q,w,x)
\end{equation}
for all $M,N\in \Sy$, $p,q\in\R^n$, $z,w\in \R$, and $x\in \Omega$. Suppose $f\in C(\Omega)$ and $u\in C(\bar{\Omega})$ are such that $u$ is a subsolution of the equation
\begin{equation}\label{eq:F-and-u}
F(D^2u,Du,u,x) = f \inOmega,
\end{equation}
and $g\in C(\Omega)$ and $v \in C(\bar{\Omega})$ are such that $v$ is a subsolution of 
\begin{equation}\label{eq:G-and-v}
G(D^2v,Dv,v,x)= g \inOmega.
\end{equation}
Then the function $w:=u+v$ is a subsolution of the equation
\begin{equation}\label{eq:H-and-w}
H(D^2w,Dw,w,x) = f+g \inOmega.
\end{equation}
Likewise, if we reverse the inequality in \EQ{H-geq-F-plus-G} and assume that $u$ and $v$ are supersolutions of \EQ{F-and-u} and \EQ{G-and-v}, respectively, then $w$ is a supersolution of \EQ{H-and-w}.
\end{lem}
\begin{proof}
Select a test function $\varphi\in C^2(\Omega)$ such that
\begin{equation}\label{eq:u-v-subsolution-localmax}
x \mapsto w(x) - \varphi(x) \quad \mbox{has a strict local maximum at } x=x_0\in \Omega.
\end{equation}
We must show
\begin{equation*}
H(D^2\varphi(x_0),D\varphi(x_0), w(x_0),x_0) \leq f(x_0)+g(x_0).
\end{equation*}
Suppose on the contrary that
\begin{equation}\label{eq:u-v-subsolution-contradict}
H(D^2\varphi(x_0),D\varphi(x_0), w(x_0),x_0) -f(x_0) - g(x_0) > 0.
\end{equation}
Define $\tilde{\varphi} = \varphi - v$. We claim that if $\eta >0$ is sufficiently small, then $\tilde{\varphi}$ is a viscosity supersolution of
\begin{equation}\label{eq:viscosityconsistency}
F(D^2\tilde{\varphi},D\tilde{\varphi}, u(x),x) \geq f \quad \mbox{in} \quad B(x_0,\eta).
\end{equation}
Select a smooth test function $\psi$ such that
\begin{equation*}
x \mapsto \tilde{\varphi}(x) - \psi(x) \quad \mbox{has a local minimum at } x=x_1\in B(x_0,\eta).
\end{equation*}
Then
\begin{equation*}
x \mapsto v(x) - ( \varphi(x)-\psi(x)) \quad \mbox{has a local maximum at } x=x_1.
\end{equation*}
Since $v$ satisfies \EQ{G-and-v}, we deduce
\begin{equation*}
G(D^2\varphi(x_1)-D^2\psi(x_1),D\varphi(x_1)-D\psi(x_1),v(x_1),x_1) \leq g(x_1).
\end{equation*}
Using \EQ{H-geq-F-plus-G}, we deduce
\begin{equation*}
\begin{aligned}
F(D^2 \psi(x_1), D\psi(x_1), u(x_1),x_1) & \geq H(D^2\varphi(x_1),D\varphi(x_1),u(x_1)+v(x_1),x_1)\\
& \qquad -G(D^2\varphi(x_1)-D^2\psi(x_1),D\varphi(x_1)-D\psi(x_1),v(x_1),x_1) \\
& \geq H(D^2\varphi(x_1),D\varphi(x_1),v(x_1)+u(x_1),x_1) - g(x_1).
\end{aligned}
\end{equation*}
Recalling \EQ{u-v-subsolution-contradict}, by the continuity of $H$ we may choose $\eta >0$ sufficiently small such that
\begin{equation*}
H(D^2\varphi(y),D\varphi(y),v(y)-u(y),y) -f(y)-g(y) > 0
\end{equation*}
for every $y \in B(x_0,\eta)$. Thus
\begin{equation*}
F(D^2 \psi(x_1), D\psi(x_1), u(x_1),x_1) > f(x_1).
\end{equation*}
It follows that $\tilde{\varphi}$ is a viscosity supersolution of \EQ{viscosityconsistency}. Since the operator $\tilde{F}$ given by
\begin{equation*}
\tilde{F}(M,p,z,x) = F(M,p,u(x),x)
\end{equation*}
is proper, we may apply the comparison principle to deduce 
\begin{equation*}
\sup_{B(x_0,\eta)} (u - \tilde{\varphi}) \leq \sup_{\partial B(x_0,\eta)} (u - \tilde{\varphi}).
\end{equation*}
Therefore, the function $w-\varphi = u-\tilde{\varphi}$ does not have a strict local minimum at $x=x_0$, a contradiction to \EQ{u-v-subsolution-localmax}. This completes the proof of the first statement. The proof of the second statement is similar. 
\end{proof}

A sufficient condition for a linear elliptic operator $L$ to satisfy the maximum principle in $\Omega$ is the existence of a supersolution $u$ of the equation $Lu=0$ that is positive on $\bar{\Omega}$ (see, e.g., \cite{Berestycki:1994}). The next result is a generalization of this observation to fully nonlinear operators $F$ satisfying our hypotheses. In fact, it contains much more information. \THM{BNV} was first discovered in the linear case by Berestycki, Nirenberg and Varadhan~\cite{Berestycki:1994} and generalized to convex, homogeneous nonlinear operators in \cite{Quaas:2008}. It is a deep result which provides a connection between the maximum principle and principal eigenvalues of elliptic operators, and we will use it many times in this paper. Simple modifications of the argument found in \cite{Quaas:2008} allow us to drop the assumption that $F$ is convex in $(M,p,z)$, and permit us to consider nonzero $f$. (See also Theorem 5.3 of \cite{Ishii:preprint}.)

\begin{thm}\label{thm:BNV}
Suppose $u,v\in C(\bar{\Omega})$ and $f\in C(\Omega)$ satisfy
\begin{equation}\label{eq:teh}
F(D^2u, Du, u, x) \leq f \leq F(D^2v,Dv,v,x) \inOmega
\end{equation}
and that $u(\tilde{x}) > v(\tilde{x})$ for some $\tilde{x}\in \Omega$. Assume also that one of the following conditions holds:
\begin{enumerate}
\item \label{enum:teh-1}
$f \leq 0$ and $u < 0$ in $\Omega$, $v \geq 0$ on $\partial \Omega$,
\end{enumerate}
or
\begin{enumerate}
\addtocounter{enumi}{1}
\item \label{enum:teh-2}
$f \geq 0$ and $v > 0$ in $\Omega$, $u \leq 0$ on $\partial \Omega$.
\end{enumerate}
Then $v\equiv tu$ for some $t>0$.
\end{thm}

\begin{proof}
We only provide a proof under the assumption \ENUM{teh-1} holds, since the proof assuming \ENUM{teh-2} is similar. By \HYP{Felliptic}, \HYP{Fhomogeneous}, and \LEM{u-v-subsolution}, for each $s \geq 1$ the function $w_s = su - v$ satisfies
\begin{equation}\label{eq:ws}
\pucci(D^2w_s) - \consp|Dw_s| - \consz |w_s| \leq sf - f \leq 0 \inOmega.
\end{equation}
We claim the function $w_s < 0$ in $\Omega$ for sufficiently large $s\geq 1$. To demonstrate this, we use a standard corollary of the Alexandrov-Bakelman-Pucci inequality (see Proposition 2.12 of \cite{Caffarelli:1996}), which is a comparison result for small domains: for any nonlinear operator $G$ which is continuous and satisfies \HYP{Felliptic}, there exists a constant $\kappa(G) > 0$ such that if $\Omega'\subseteq \Omega$ is an open subset of $\Omega$ so small that $| \Omega' | \leq \kappa(G)$, then the maximum and minimum principles hold for $G$ in $\Omega'$ (see \cite{Quaas:2008} for a proof). Now select a compact subset $K\subseteq \Omega$ so large that
\begin{equation*}
| \Omega \backslash K | \leq \kappa\left( \pucci(D^2\cdot) - \consp|D\cdot| - \consz| \cdot| \right).
\end{equation*}
We may choose $s\geq 1$ large enough that $w_s < 0$ on $K$. Owing to our hypotheses, $w_s \leq 0$ on $\partial \Omega$, and thus
\begin{equation*}
w_s \leq 0 \quad \mbox{on} \quad \partial (\Omega \backslash K).
\end{equation*}
According to our choice of $\kappa$, we may apply the maximum principle to deduce $w_s \leq 0$ in $\Omega\backslash K$ and hence in $\Omega$. Since $w_s \not\equiv 0$, Hopf's Lemma (\THM{hopf}) implies $w_s < 0$ in $\Omega$. Now define
\begin{equation}\label{eq:choiceoft}
t = \inf \{ s : w_s < 0 \mbox{ in } \Omega\}.
\end{equation}
Since $u(\tilde{x}) > v(\tilde{x})$ we have $t > 1$. We claim $w_t \equiv 0$. If not, then $w_t < 0$ in $\Omega$ by Hopf's Lemma. Then we may find a small number $0 < \sigma < t-1$ such that $w_{t-\sigma} < 0$ on $K$. Repeating the argument above, we discover $w_{t-\sigma} < 0$ in $\Omega$, in contradiction to \EQ{choiceoft}. Thus $w_t\equiv 0$, and so $v \equiv tu$.
\end{proof}

\begin{rem}\label{rem:BNV-half-homogeneous}
We may relax our hypothesis \HYP{Fhomogeneous} and still maintain either \ENUM{teh-1} or \ENUM{teh-2} of \THM{BNV}. Indeed, an inspection of the proof above reveals that for \THM{BNV}\ENUM{teh-1}, we need require only
\begin{equation*}
F(tM,tp,tz,x) \leq t F(M,p,z,x) \quad \mbox{for all}\quad t \geq 0,
\end{equation*}
and for \ENUM{teh-2} we need only the reverse of this inequality.
\end{rem}

\medskip

For each $\lambda \in \R$, define a nonlinear operator $G_\lambda$ by
\begin{equation*}
G_\lambda(M,p,z,x) = F(M,p,z,x) - \lambda z.
\end{equation*}
The following is an immediate consequence of \THM{BNV}.
\begin{cor}\label{cor:BNV}
If there exists a solution $v\in C(\bar{\Omega})$ of 
\begin{equation*}
F(D^2v,Dv,v,x) - \lambda v \geq (\leq) \ 0 \inOmega
\end{equation*}
for which $v > 0$ ($v < 0$) in $\Omega$ and $v\not\equiv 0$ on $\partial \Omega$, then $G_\lambda$ satisfies the maximum (minimum) principle in $\Omega$.
\end{cor}
From \THM{BNV} we also obtain the following one-sided comparison principle.
\begin{cor}\label{cor:CP-onesided}
Assume $\lambda \in \R$ and $f\in C(\Omega)$ is such that $f\not \equiv 0$. Suppose $u,v\in C(\bar{\Omega})$ satisfy 
\begin{equation}\label{eq:DP-uniqueness-minus}
F(D^2u,Du,u,x) - \lambda u \leq f \leq F(D^2v,Dv,v,x) - \lambda v \inOmega
\end{equation}
and that $u \leq v$ on $\partial \Omega$. Assume also that either \emph{(i)} $f \leq 0$ and $u \leq 0$ in $\Omega$, or  \emph{(ii)} $f \geq 0$ and $v \geq 0$ in $\Omega$. Then $u \leq v$ in $\Omega$.
\end{cor}

Define constants
\begin{align*}
& \mu^+(F,\Omega) = \sup \left\{ \rho \, : \,  G_\rho \mbox{ satisfies the maximum principle in } \Omega \right\}, \intertext{and}
& \mu^-(F,\Omega)  = \sup \left\{ \rho \, : \,  G_\rho \mbox{ satisfies the minimum principle in } \Omega \right\}.
\end{align*}

We will eventually show $\lambda^\pm_1(F,\Omega) = \mu^\pm(F,\Omega)$. The following lemma is the first step in this direction.

\begin{lem} \label{lem:lambda-leq-mu}
Let $\consz \geq 0$ be as in \HYP{Felliptic}. Then
\begin{equation}\label{eq:lambda-leq-mu}
-\consz \leq \lambda^{\pm}_1(F,\Omega) \leq \mu^{\pm} (F,\Omega) < \infty.
\end{equation}
\end{lem}
\begin{proof}
To see that $-\consz \leq \lambda^{\pm}_1(F,\Omega)$, notice that for every $x\in \Omega$
\begin{equation*}
G_{-\consz}(0,0,-1,x) \leq 0 \leq G_{-\consz}(0,0,1,x).
\end{equation*}
We will now show
\begin{equation}\label{eq:lambda-leq-mu-1}
\lambda^+_1(F,\Omega) \leq \mu^+(F,\Omega).
\end{equation}
Suppose on the contrary $\mu^+(F,\Omega) < \rho_1 < \rho_2 < \lambda^+_1(F,\Omega)$. Then we may select a function $v_1\in C(\bar{\Omega})$ which satisfies
\begin{equation*}
F(D^2v_1,Dv_1,v_1,x) - \rho_1 v_1 \leq 0 \inOmega
\end{equation*}
and such that $v_1\leq 0$ on $\partial \Omega$ and $v_1 > 0$ somewhere in $\Omega$. We can also select $v_2\in C(\bar{\Omega})$ such that $v_2 > 0$ in $\Omega$ and $v_2$ satisfies
\begin{equation*}
F(D^2v_2,Dv_2,v_2,x) -\rho_2 v_2 \geq 0 \inOmega.
\end{equation*}
Since $-\rho_1 v_2 > -\rho_2v_2$, we may apply \THM{BNV} to deduce $v_2 \equiv tv_1$ for some $t>0$. This implies $\rho_1=\rho_2$, a contradiction which establishes \EQ{lambda-leq-mu-1}. The inequality $\lambda^-_1(F,\Omega) \leq \mu^-(F,\Omega)$ is demonstrated via a similar argument.

Finally, we will show the operator $G_\rho$ does not satisfy the minimum principle in $\Omega$ for all sufficiently large $\rho$. By replacing $F$ with $G_{-\consz}$, if necessary, we may assume $F$ is proper. Select a continuous function $h\leq 0$, $h\not\equiv 0$ with compact support in $\Omega$. Standard results (see, e.g., \cite{Crandall:1999,Crandall:2000,Winter:2008}) imply the existence of a solution $v\in C^{1,\alpha}(\Omega)$ of the Dirichlet problem
\begin{equation*}
\left\{ \begin{aligned}
F(D^2v,Dv,v,x) = {}& h & \mbox{in} & \quad \Omega \\
v = {}& 0 & \mbox{on} & \quad \partial \Omega.
\end{aligned}\right.
\end{equation*}
According to the comparison principle, $v \leq 0$ in $\Omega$. Since $h\not\equiv 0$, we have $v\not\equiv 0$. Thus $v < 0$ in $\Omega$ according to Hopf's Lemma. Since $h$ has compact support in $\Omega$ we may select a constant $\rho_0 > 0$ such that $\rho_0 v \leq h$. Therefore, $v$ is a supersolution of the PDE
\begin{equation*}
F(D^2v,Dv,v,x) - \rho_0 v \geq 0 \inOmega,
\end{equation*}
and so evidently the operator $G_\rho$ does not satisfy the minimum principle in $\Omega$, for any $\rho \geq \rho_0$. Thus $\mu^-(F,\Omega) \leq \rho_0$. Via a similar argument, we verify that $\mu^+(F,\Omega) < \infty$.
\end{proof}

From the proof of \LEM{lambda-leq-mu} we also deduce that
\begin{equation}\label{eq:lambda-mu-interval-subset-p}
( -\infty , \lambda^+_1(F,\Omega) ) \subseteq \left\{ \rho \, : \,  G_\rho \mbox{ satisfies the maximum principle in } \Omega \right\}\\
\end{equation}
and
\begin{equation}
\label{eq:lambda-mu-interval-subset-m}
( -\infty , \lambda^-_1(F,\Omega) ) \subseteq \left\{ \rho \, : \,  G_\rho \mbox{ satisfies the minimum principle in } \Omega \right\}.
\end{equation}
We will see later that we have equality in \EQ{lambda-mu-interval-subset-p} and \EQ{lambda-mu-interval-subset-m}.

\medskip

We are now ready to show our operator $F$ has two principal half-eigenvalues. Instead of invoking the Krein-Rutman theorem, we choose instead a proof based on the Leray-Schauder Alternative principle, which is also called Schaeffer's Fixed Point Theorem. For the reader's convenience, we will first state this result.

\begin{definition}
If $X$ and $Y$ are Banach spaces, we say a (possibly nonlinear) map $\mathcal{A}:X \to Y$ is \emph{compact} if, for each bounded subset $B \subseteq X$, the closure of the set $\{ \mathcal{A}(x) : x \in B \}$ is compact in $Y$.
\end{definition}

\begin{thm}[Leray-Schauder Alternative]\label{thm:leray-schauder-alt}
Suppose $X$ is a Banach space, and $C \subseteq X$ is a convex subset of $X$ such that $0 \in C$. Assume $\mathcal{A}:C \to C$ is a (possibly nonlinear) function which is compact and continuous. Then at least one of the following holds:
\begin{enumerate}
\item the set $\{ x \in C \, : \, x = \mu \mathcal{A}(x) \mbox{ for some } 0 < \mu < 1 \}$ is unbounded in $X$,
\end{enumerate}
or
\begin{enumerate}
\addtocounter{enumi}{1}
\item there exists $x\in C$ for which $x = \mathcal{A}(x)$.
\end{enumerate}
\end{thm}
See Theorem 5.4 on page 124 of \cite{Granas:Book} for a proof of \THM{leray-schauder-alt}. The following argument is a straightforward adaptation of that found in Section~6.5.2 of \cite{Evans:Book}.

\begin{proof}[Proof of \THM{eigenvalues}]
We assume without loss of generality that $F$ is proper, since otherwise we may consider the operator $G_{-\consz}$ in place of $F$. For each $v\in C(\bar{\Omega})$, define $u=\mathcal{A}(v)$ to be the unique solution $u \in C^{1,\alpha}(\Omega)$ of the Dirichlet problem
\begin{equation*}
\left\{ \begin{aligned}
F(D^2u,Du,u,x) = {}& v & \mbox{in} & \quad \Omega \\
u = {}& 0 & \mbox{on} & \quad \partial \Omega.
\end{aligned}\right.
\end{equation*}
We claim
\begin{equation}\label{eq:eigenvalues-claim-1}
\mathcal{A}:C(\bar{\Omega}) \rightarrow C(\bar{\Omega}) \quad \mbox{is a continuous, compact operator.}
\end{equation}
Let $u_1 = \mathcal{A}(v_1)$ and $u_2 = \mathcal{A}(v_2)$ and notice the function $w = u_1 - u_2$ satisfies
\begin{equation*}
\pucci(D^2 w) - \consp|Dw| \leq v_1 - v_2 \quad \mbox{in} \quad \{ w > 0 \}.
\end{equation*}
The Alexandrov-Bakelman-Pucci inequality implies
\begin{equation*}
w \leq C \| v_1 - v_2\|_{L^n(\Omega)} \leq C \| v_1 - v_2 \|_{C(\bar{\Omega})}.
\end{equation*}
Reversing the roles of $u_1$ and $u_2$ we obtain
\begin{equation*}
\| u_1 - u_2\|_{C(\bar{\Omega})} \leq C \| v_1-v_2\|_{C(\bar{\Omega})}.
\end{equation*}
Moreover, the $C^{1,\alpha}$ estimates for uniformly elliptic equations (see \cite{Trudinger:1988} and \cite{Winter:2008}) imply 
\begin{equation*}
\| \mathcal{A}(v)\|_{C^{1,\alpha}(\Omega)} \leq C \left( \|\mathcal{A}(v)\|_{L^\infty(\Omega)} + \| v \|_{L^\infty(\Omega)}\right) \leq C \| v \|_{C(\bar{\Omega})}.
\end{equation*}
We have demonstrated \EQ{eigenvalues-claim-1}.

Let $C\subseteq X$ be the cone $C=\left\{ v\in C(\bar{\Omega})  :  v \geq 0 \right\}$ of nonnegative continuous functions on $\bar{\Omega}$. According to the maximum principle,
\begin{equation*}
\mathcal{A}:C\rightarrow C.
\end{equation*}
Select a nonzero $h\in C$ which has compact support in $\Omega$. We now claim
\begin{equation}\label{eq:eigenvalues-claim-2}
\mbox{if } u\in C \mbox{ satisfies } u = \lambda \mathcal{A}(u+h) \mbox{ then } \lambda \leq \lambda_1^+(F,\Omega).
\end{equation}
Suppose $u$ and $\lambda$ satisfy the hypothesis of \EQ{eigenvalues-claim-2}. According to Hopf's Lemma, $u> 0$ in $\Omega$. Thus
\begin{equation*}
\lambda \in  \{ \, \rho : \exists v\in C(\Omega) \mbox{ such that } v > 0 \mbox{ and } G_\rho(D^2v,Dv,v,x) \geq 0 \mbox{ in } \Omega\},
\end{equation*}
and so $\lambda \leq \lambda^+_1(F,\Omega)$. This confirms \EQ{eigenvalues-claim-2}.

\medskip

We now apply \THM{leray-schauder-alt} to deduce that the set 
\begin{equation*}
D_\varepsilon = \left\{ \, u \in C: \mbox{there exists } 0 \leq \lambda \leq \lambda_1^+(F,\Omega)  + \varepsilon \mbox{ such that } u = \lambda \mathcal{A}(u+\varepsilon h) \right\}
\end{equation*}
is unbounded in $C(\bar{\Omega})$, for every $\varepsilon > 0$. Therefore, we may select sequences $\{ u_\varepsilon\} \subseteq C$ and $\{ \lambda_\varepsilon  \} \subseteq [0,\lambda_1^+(F,\Omega)  + \varepsilon]$ such that $\|u_\varepsilon\|_{C(\bar{\Omega})} \geq 1$ and $u_\varepsilon = \lambda_\varepsilon \mathcal{A}(u_\varepsilon + \varepsilon h)$. Normalize by setting $v_\varepsilon = u_\varepsilon / \|u_\varepsilon\|_{C(\bar{\Omega})}$. Then
\begin{equation}\label{eq:vepsilon}
v_\varepsilon = \lambda_\varepsilon \mathcal{A}(v_\varepsilon+ \varepsilon h / \| u_\varepsilon\|_{C(\bar{\Omega})}).
\end{equation}
Since $\mathcal{A}$ is a compact map, we can find $\varphi^+_1\in C$ and $\lambda^* \in \left[ 0, \lambda_1^+(F,\Omega)\right]$ and a subsequence $\varepsilon_k \rightarrow 0$ such that $v_{\varepsilon_k} \rightarrow \varphi^+_1$ uniformly on $\bar{\Omega}$ and $\lambda_{\varepsilon_k} \rightarrow \lambda^*$, as $k \rightarrow \infty$. Passing to limits in \EQ{vepsilon}, we have
\begin{equation*}
\varphi^+_1= \lambda^* \mathcal{A}(\varphi^+_1).
\end{equation*}
Notice that $\| \varphi^+_1\|_{C(\bar{\Omega})} = \lim \|v_k\|_{C(\bar{\Omega})} = 1$, and so $\varphi^+_1 \not\equiv 0$. According to Hopf's Lemma, $\varphi^+_1 > 0$ in $\Omega$. It is now immediate from the definitions of $\mu^+(F,\Omega)$ and $\lambda^+_1(F,\Omega)$ that
\begin{equation*}
\mu^+(F,\Omega) \leq \lambda^* \leq \lambda_1^+(F,\Omega),
\end{equation*}
and therefore $\lambda^* = \mu^+(F,\Omega) = \lambda_1^+(F,\Omega)$, by \LEM{lambda-leq-mu}. In fact, we have equality in \EQ{lambda-mu-interval-subset-p} and \EQ{lambda-mu-interval-subset-m}.

It is evident that the function $\varphi^+_1$ satisfies \EQ{Feigenvalues}. The simplicity and uniqueness of the eigenvalue $\lambda_1^+(F,\Omega)$ are immediately obtained from \THM{BNV}. This completes the proof of all assertions for $\lambda_1^+(F,\Omega)$ and $\varphi^+_1$. Making appropriate modifications to our arguments, we obtain the corresponding assertions for $\lambda^-_1(F,\Omega)$ and $\varphi^-_1$.
\end{proof}

We henceforth adopt the normalization
\begin{equation}\label{eq:max-eigenfunctions-equal-1}
\sup_\Omega \varphi^+_1 = \sup_\Omega \left(- \varphi^-_1\right) = 1.
\end{equation}
We will sometimes write $\varphi^\pm_1 = \varphi^\pm_1(F,\Omega)$ for clarity.
The $C^{1,\alpha}$ estimates now assert that
\begin{equation}\label{eq:eigenfunction-estimates}
\| \varphi^{\pm}_1 \|_{C^{1,\alpha}(\Omega)} \leq C,
\end{equation}
where the constant $C$ depends only on $\Omega,\gamma,\Gamma,\consp,\consz,n,p$ and $\lambda^\pm_1(F,\Omega)$. According to \COR{BNV}, 
\begin{equation}
\Omega' \subsetneq \Omega \quad\mbox{implies that} \quad \lambda^\pm_1(F,\Omega) < \lambda^\pm_1(F,\Omega'),
\end{equation}
since the positive (negative) principal eigenfunction for $F$ in $\Omega$ must be positive (negative) somewhere on $\partial \Omega'$.

\begin{exam}
As a consequence of \THM{BNV}, if $\lambda_1^+(F,\Omega)=\lambda_1^-(F,\Omega)$, then the positive and negative principal eigenfunctions are proportional: $\varphi^+_1 \equiv - \varphi^-_1$. However, the converse is not true. For example, consider the operator
\begin{equation*}
G(D^2u) = \min \{ -\Delta u, -2\Delta u \}. 
\end{equation*}
It is simple to see that
\begin{equation*}
\lambda^+_1(G,\Omega) = \lambda_1(-\Delta,\Omega) < 2\lambda_1(-\Delta,\Omega) = \lambda_1^-(G,\Omega),
\end{equation*}
but $\varphi^+_1 \equiv - \varphi^-_1$, since the principal eigenfunction of the Laplacian $-\Delta$ in $\Omega$ is proportional to both of the principal eigenfunctions of $G$.

For an example of an operator $F$ for which $\varphi^+_1(F,\Omega) \not\equiv - \varphi^-_1(F,\Omega)$, take $F = \pucci$. A straightforward calculation convinces us that if $\varphi^+_1 \equiv - \varphi^-_1$, then $\varphi^+_1$ is both concave and equal to the principal eigenfunction of $-\Delta$ in $\Omega$, which is impossible. For details, see \cite{Quaas:2008}.
\end{exam}

\begin{exam}
Consider the Bellman operator
\begin{equation}\label{eq:bellman-operator}
H(D^2u,Du,u,x) = \sup_{\alpha \in A} L^\alpha u,
\end{equation}
where $\{ L^\alpha : \alpha \in A\}$ is a family of linear operators having the form
\begin{equation*}
L^\alpha u = -a^{ij}_\alpha u_{ij} + b^j_\alpha u_j + c_\alpha u,
\end{equation*} 
with sufficient hypotheses on the coefficients such that $H$ satisfies \HYP{Fcontinuous}, \HYP{Felliptic}, and \HYP{Fhomogeneous}. It is immediate that
\begin{equation}\label{eq:bellman-eig-inequalities}
\lambda^-_1(H,\Omega) \leq \inf_{\alpha \in A} \lambda_1(L^\alpha,\Omega) \leq \sup_{\alpha \in A} \lambda_1(L^\alpha,\Omega) \leq \lambda^+_1(H,\Omega).
\end{equation}
In general, we cannot expect equality in either the first or the last inequality of \EQ{bellman-eig-inequalities}. For an explicit example, let $A = \{ 1,2\}$ and $L^1$ and $L^2$ be linear elliptic operators with smooth coefficients such that $\lambda_1(L^1,\Omega) = \lambda_1(L^2,\Omega)=:\lambda$, but $\varphi_1:= \varphi^+_1(L^1,\Omega) \not\equiv \varphi^+_1(L^2,\Omega) =:\varphi_2$. We claim
\begin{equation*}
\lambda^-_1(H,\Omega) < \lambda < \lambda^+_1(H,\Omega).
\end{equation*}
If, on the contrary $\lambda^-_1(H,\Omega) = \lambda$, then we may appeal to \THM{BNV} to deduce that $-\varphi_1 \equiv \varphi^-_1(H,\Omega) \equiv -\varphi_2$, in violation of our assumption that $\varphi_1\not\equiv \varphi_2$.
Similarly, if $\lambda^+_1(H,\Omega) = \lambda$, then $\varphi_1 \equiv \varphi^+_1(H,\Omega) \equiv \varphi_2$. 

This observation corrects a mistake in \cite[Theorem 1.1]{Busca:1999}, which incorrectly implies that the first inequality in \EQ{bellman-eig-inequalities} is an equality in general. (This error was also repeated in \cite{Quaas:2008}.) 
\end{exam}

\medskip

We now show that $F$ has no eigenvalues between $\lambda_1^+(F,\Omega)$ and $\lambda_1^-(F,\Omega)$, or less than $\minboth$. In the case that $F$ is convex (or concave) in $(M,p,z)$, this was demonstrated previously in \cite{Quaas:2008}.

\begin{lem}\label{lem:no-eigenvalues-in-between}
Suppose $\rho < \maxboth$ is an eigenvalue of $F$ in $\Omega$. Then $\rho = \minboth$.
\end{lem}
\begin{proof}
Suppose that $w \in C(\bar{\Omega})$ is a nontrivial solution of the problem
\begin{equation*}
\left\{ \begin{aligned}
F(D^2w,Dw,w,x) = {}& \rho w & \mbox{in} & \quad \Omega \\
w = {}& 0 & \mbox{on} & \quad \partial \Omega.
\end{aligned}\right.
\end{equation*}
If $\rho < \lambda_1^-(F,\Omega)$, the operator $G_\rho$ satisfies the minimum principle. Therefore $w \geq 0$ in $\Omega$. According to Hopf's Lemma, $w > 0$. By the uniqueness of the positive principal eigenvalue, it follows that $\rho = \lambda_1^+(F,\Omega)$. Arguing similarly, we deduce that if $\rho < \lambda^+_1(F,\Omega)$, then $\rho = \lambda^-_1(F,\Omega)$.
\end{proof}

\begin{lem}\label{lem:lambda-2-is-eigenvalue}
If $\lambda_2(F,\Omega) < \infty$, then there is a nonzero function $\varphi_2 \in C^{1,\alpha}(\Omega)$ which solves the Dirichlet problem
\begin{equation}\label{eq:eigenvalue-2}
\left\{ \begin{aligned}
F(D^2\varphi_2,D\varphi_2,\varphi_2,x) = {}& \lambda_2(F,\Omega) \varphi_2 & \mbox{in} & \quad \Omega \\
\varphi_2 = {}& 0 & \mbox{on} & \quad \partial \Omega.
\end{aligned}\right.
\end{equation}
\end{lem}
\begin{proof}
Take a sequence $\rho_k \to \lambda_2(F,\Omega)$ of eigenvalues, with corresponding eigenfunctions $u_k \in C(\bar{\Omega})$, satisfying the problem
\begin{equation} \label{eq:eigenvalue-2-k}
\left\{ \begin{aligned}
F(D^2u_k,Du_k,u_k,x) = {}& \rho_k u_k & \mbox{in} & \quad \Omega \\
u_k = {}& 0 & \mbox{on} & \quad \partial \Omega,
\end{aligned}\right.
\end{equation}
and subject to the normalization $\| u_k \|_{L^\infty(\Omega)} = 1$. We have the estimate
\begin{equation*}
\| u_k \|_{C^{1,\alpha}(\Omega)} \leq C.
\end{equation*}
Thus, up to a subsequence, $u_k$ converges to a function $\varphi_2 \in C(\bar{\Omega})$ uniformly on $\bar{\Omega}$. Noticing that $\| \varphi_2 \|_{L^\infty(\Omega)} = 1$, and passing to limits in \EQ{eigenvalue-2-k}, we see that $\varphi_2$ is a nontrivial solution of problem \EQ{eigenvalue-2}. Now apply the $C^{1,\alpha}$ estimates once again to conclude.
\end{proof}



\section{The Dirichlet problem for $\lambda < \minboth$  } \label{sec:Dirichlet-problem}

In this section we study questions of existence and uniqueness of solutions to the Dirichlet problem \EQ{Dirichlet-problem} for $\lambda$ less than the principal eigenvalues. We begin with the proof of \THM{DP-existence}.

\begin{proof}[Proof of \THM{DP-existence}]
The uniqueness assertions in \ENUM{DP-existence-plus} and \ENUM{DP-existence-minus} follow at once from \COR{CP-onesided}. We will only show existence for \ENUM{DP-existence-plus}, since the arguments for \ENUM{DP-existence-minus} and \ENUM{DP-existence-both} are very similar. Without loss of generality, suppose $\lambda = 0 < \lambda^+_1(F,\Omega)$. We will first produce a solution under the condition $f$ has support on a compact subset of $\Omega$. In this case we may find a large constant $A > 0$ so that
\begin{equation*}
0 \leq f \leq A \lambda^+_1 \varphi^+_1 \inOmega.
\end{equation*}
Then $u^* = A\varphi^+_1$ and $u_* \equiv 0$ satisfy
\begin{equation*}
F(D^2u_*,Du_*,u_*,x) \leq f \leq F(D^2u^*,Du^*,u^*,x) \inOmega,
\end{equation*}
and $u^* = u_* = 0$ on $\partial \Omega$. Standard results (see, for example, \cite{UsersGuide}) imply the existence of $u\in C^{1,\alpha}(\Omega)$ solving the Dirichlet problem \EQ{Dirichlet-problem}, and satisfying $0 = u_*\leq u \leq u^*$ in $\Omega$.

For general nonnegative $f\in C(\Omega) \cap L^p(\Omega)$, not necessary with compact support, we take a sequence $\{ f_k \} \subset C(\Omega)$ of nonnegative continuous functions with compact support in $\Omega$ such that $f_k \to f$ uniformly on each compact subset of $\Omega$ and $f_k \to f$ in $L^p(\Omega)$. Let $u_k \geq 0$ solve \EQ{Dirichlet-problem} with $f$ replaced by $f_k$. We claim
\begin{equation}\label{eq:DP-existence-1}
\sup_{k\geq 1} \| u_k \|_{L^\infty(\Omega)} < \infty.
\end{equation}
If not, we may assume $\| u_k \|_{L^\infty(\Omega)} \to \infty$. Set $v_k = u_k / \| u_k \|_{L^\infty(\Omega)}$ and notice $v_k$ satisfies the equation
\begin{equation*}
F(D^2v_k,Dv_k,v_k,x) = f / \| u_k \|_{L^\infty(\Omega)} \inOmega.
\end{equation*}
We have the estimate
\begin{equation*}
\| v_k \|_{C^{1,\alpha}(\Omega)} \leq C.
\end{equation*}
Taking a subsequence, if necessary, we may assume $v_k$ converges to a function $v \in C^{1,\alpha}(\Omega)$ uniformly on $\bar{\Omega}$. Passing to limits, we see $v$ satisfies the equation
\begin{equation*}
F(D^2v,Dv,v,x) = 0 \inOmega.
\end{equation*}
Moreover, $v = 0$ on $\partial \Omega$. Since $\lambda^+_1(F,\Omega) > 0$, the nonlinearity $F$ satisfies the maximum principle in $\Omega$, and so $v \equiv 0$. However, $\| v \|_{L^\infty(\Omega)} = \lim \| v_k \|_{L^\infty(\Omega)} = 1$. This contradiction establishes \EQ{DP-existence-1}. 

Now we apply the $C^{1,\alpha}$ estimates once again to obtain
\begin{equation*}
\| u_k \|_{C^{1,\alpha}(\Omega)} \leq C.
\end{equation*}
We may select a function $u \in C^{1,\alpha}(\Omega)$ such that, up to a subsequence, $u_k \to u$ uniformly on $\bar{\Omega}$. Obviously, $u \geq 0$ in $\Omega$. Now we pass to limits to conclude that $u$ is a solution of \EQ{Dirichlet-problem}.
\end{proof}

\begin{rem}
By simple modifications of the previous argument, we deduce similar results for the inhomogeneous Bellman-Isaacs problem
\begin{equation}\label{eq:Isaacs-inhomogeneous}
\left\{ \begin{aligned}
\inf_{\alpha} \sup_\beta \left\{ L^{\alpha,\beta}u - f^{\alpha,\beta} \right\} = {} & 0 & \mbox{in} & \quad \Omega \\
u = {}& 0 & \mbox{on} & \quad \partial \Omega.
\end{aligned}\right.
\end{equation}
If $F$ is given by \EQ{BIoperator} and satisfies \HYP{Fcontinuous}, \HYP{Felliptic} and \HYP{Fhomogeneous}, then for any uniformly bounded family $\left\{ f^{\alpha,\beta}\right\} \subseteq C(\bar{\Omega})$, the Dirichlet problem \EQ{Isaacs-inhomogeneous} has a solution $u\in C^{1,\alpha}(\Omega)$ provided
\begin{equation*}
\minboth > 0.
\end{equation*}
Likewise, if each function in the family $\left\{ f^{\alpha,\beta}\right\}$ is nonnegative (nonpositive), and $\lambda^+_1(F,\Omega) > 0$ ($\lambda^-_1(F,\Omega) > 0$), then the problem \EQ{Isaacs-inhomogeneous} has a nonnegative (nonpositive) solution $u \in  C^{1,\alpha}(\Omega)$. By \REM{BNV-half-homogeneous}, we see that this nonnegative (nonpositive) solution is unique.
\end{rem}

\medskip

We now discuss sufficient conditions for the solutions obtained in \THM{DP-existence}\ENUM{DP-existence-both} to be unique. To this end, and following \cite{Ishii:preprint}, we define nonlinear operators $\Fup$ and $\Fdown$ by 
\begin{align}
& \Fup(M,p,z,x) = \sup\left\{ F(M+N,p+q,z+w,x) - F(N,q,w,x) \right\} \label{eq:Fup}
\intertext{and}
& \Fdown(M,p,z,x) = \inf\left\{ F(M+N,p+q,z+w,x) - F(N,q,w,x) \right\}, \label{eq:Fdown}
\end{align}
where the supremum and infimum above are taken over $(N,q,w) \in \Sy\times\R^n \times\R$. If $G$ and $H$ are arbitrary real-valued functions on $\Sy \times \R^n \times \R \times \Omega$ satisfying
\begin{equation*}
G(M-N,p-q,z-w,x) \leq F(M,p,z,x) - F(N,q,w,x) \leq H(M-N,p-q,z-w,x),
\end{equation*}
then $G \leq \Fdown \leq \Fup \leq H$. The next proposition summarizes useful properties of $\Fup$ and $\Fdown$ as they relate to $F$.

\begin{prop}\label{prop:starstar}
The operators $\Fup$ and $\Fdown$ have the following properties:
\begin{enumerate}
\item $\Fup$ and $\Fdown$ satisfy \HYP{Fcontinuous}, \HYP{Felliptic}, and \HYP{Fhomogeneous}; \label{enum:starprops}
\item for each $x\in \Omega$, the map $(M,p,z) \mapsto \Fup(M,p,z,x)$ is convex, while the map $(M,p,z) \mapsto \Fdown(M,p,z,x)$ is concave; \label{enum:starconvex}
\item for all $M,N \in \Sy$, $p,q\in \R^n$, $z,w\in \R$, and $x\in \Omega$, 
\begin{multline*}
\qquad \Fdown(M-N,p-q,z-w,x) \leq F(M,p,z,x) - F(N,q,w,x)  \\
\leq \Fup(M-N,p-q,z-w,x);
\end{multline*}\label{enum:starbound}
\item for all $M\in \Sy$, $p\in \R^n$, $z\in\R$ and $x\in \Omega$, 
\begin{equation*}
\Fup(M,p,z,x) = -\Fdown(-M,-p,-z,x);
\end{equation*}\label{enum:starstar}
\item $F = F^*$ if and only if $(M,p,z)\mapsto F(M,p,z,x)$ is convex, while $F=F_*$ if and only if $(M,p,z) \mapsto F(M,p,z,x)$ is concave; and\label{enum:starnostar}
\item The principal eigenvalues of $\Fup$ and $\Fdown$ are related to those of $F$ by
\begin{equation*}
\lambda^-_1(\Fup,\Omega) = \lambda^+_1(\Fdown,\Omega) \leq \lambda^+_1(F,\Omega),\lambda^-_1(F,\Omega) \leq \lambda^+_1(\Fup,\Omega) = \lambda^-_1(\Fdown,\Omega).
\end{equation*}\label{enum:stareigenvalues}
\end{enumerate}
\end{prop}
\begin{proof}
The properties \ENUM{starprops}, \ENUM{starbound} and \ENUM{starstar} are immediate from the definitions \EQ{Fup} and \EQ{Fdown}, and \ENUM{stareigenvalues} follows at once from these. To deduce \ENUM{starconvex}, observe
\begin{align*}
\lefteqn{\Fup(M_1,p_1,z_1,x) + \Fup(M_2,p_2,z_2,x)} \quad & \\
& = \inf\left\{ F(M_1+N_1,p_1+q_1,z_1+w_1,x) - F(N_1,q_1,w_1,x) \right\} \\ 
& \qquad + \inf\left\{ F(M_2+N_2,p_2+q_2,z_2+w_2,x) - F(N_2,q_2,w_2,x) \right\} \\
& = \inf\left\{ F(M_1+M_2+N_1,p_1+p_2+q_1,z_1+z_2+w_1,x) - F(M_2+N_1,p_2+q_1,z_2+w_1,x) \right\} \\ 
& \qquad + \inf\left\{ F(M_2+N_2,p_2+q_2,z_2+w_2,x) - F(N_2,q_2,w_2,x) \right\} \\
& \geq \inf\left\{ F(M_1+M_2+N_1,p_1+p_2+q_1,z_1+z_2+w_1,x) - F(N_1,q_1,w_1,x) \right\}\\ 
& = \Fup(M_1+M_2,p_1+p_2,z_1+z_2,x).
\end{align*}
Recall that for a positively homogeneous function, the notions of convexity and sublinearity are equivalent. Thus $\Fup$ is convex in $(M,p,z)$. From \ENUM{starstar}, we see that $\Fdown$ is concave in $(M,p,z)$, confirming \ENUM{starconvex}. It is clear by now that $F=\Fup$ if and only if $F$ is sublinear, from which \ENUM{starnostar} follows.
\end{proof}

The following result of \cite{Ishii:preprint} is also a consequence of Theorem 1.5 of \cite{Quaas:2008} once we have \PROP{starstar} above. We record its proof for completeness.

\begin{thm}[Ishii and Yoshimura \cite{Ishii:preprint}]
If $\lambda^-_1(\Fup,\Omega) > 0$, then $F$ satisfies the comparison principle in $\Omega$.
\end{thm}
\begin{proof}
Suppose that $u,v\in C(\bar{\Omega})$ and $f\in C(\Omega)$ are such that
\begin{equation*}
F(D^2u,Du,u,x) \leq f \leq F(D^2v,Dv,v,x) \inOmega,
\end{equation*}
and $u \leq v$ on $\partial \Omega$. \PROP{starstar}\ENUM{starbound} and \LEM{u-v-subsolution} imply that the function $w = v-u$ satisfies
\begin{equation*}
F^*(D^2w,Dw,w,x) \geq 0 \inOmega,
\end{equation*}
and $w \geq 0$ on $\partial \Omega$. If $\lambda^-_1(\Fup,\Omega) > 0$, then the minimum principle holds for $\Fup$ in $\Omega$. Thus $w$ is nonnegative on $\bar{\Omega}$.
\end{proof}

\begin{rem}
I do not know whether $\lambda^-_1(\Fup,\Omega)$ is equal to $\minboth$ in general, or if $\lambda^-_1(\Fup,\Omega)> 0$ is necessary for $F$ to satisfy the comparison principle in $\Omega$.
\end{rem}



\section{The Dirichlet problem for $\lambda > \max\{\lambda^-_1(F,\Omega),\lambda^+_1(F,\Omega)\}$} \label{sec:existence-past-lambda-1}

The purpose of this section is to prove \THM{existence-past-lambda-1}. Our argument will utilize the theory of Leray-Schauder degree. Our plan is to build a homotopy between our problem \EQ{Dirichlet-problem} and a similar Dirichlet problem for the Laplacian, and then argue solutions must exist along the path of the homotopy. For the convenience of the reader, we now briefly introduce the concept of Leray-Schauder degree.

\medskip

Let $X$ be a Banach space. The identity map on $X$ is denoted by $\mathcal{I}_X$. We denote the set of (possibly nonlinear) compact maps from $X$ to itself by $K(X)$, and we define the set $K_1(X)$ of compact perturbation of the identity by
\begin{equation*}
K_1 (X) = \left\{ \mathcal{I}_X + \mathcal{A} \, : \, \mathcal{A} \in K(X) \right\}.
\end{equation*}
The norm topology on $X$ is written
\begin{equation*}
\tau(X) = \left\{ W \subseteq X \, : \, W \mbox{ is open} \right\}.
\end{equation*}
We need to define an appropriate notion of homotopy.
\begin{definition}\label{def:homotopy}
We say $\mathcal{A}:X\times [0,1] \to X$ is a \emph{homotopy of compact operators on} $X$ if:
\begin{enumerate}
\item the map $u \mapsto \mathcal{A}(u,s)$ is a compact operator on $X$, for each fixed $s\in [0,1]$,
\end{enumerate}
and
\begin{enumerate}
\addtocounter{enumi}{1}
\item for every $\varepsilon > 0$ and $C > 0$, there exists $\eta > 0$ such that $|s-t| \leq \eta$ and $\| u \|_X \leq C$ imply $\| \mathcal{A}(u,s) - \mathcal{A}(u,t) \|_X \leq \varepsilon$.
\end{enumerate}
\end{definition}

\begin{thm}[Existence of Leray-Schauder Degree]\label{thm:leray-schauder-degree}
There exists an integer-valued function
\begin{equation*}
\deg : K_1(X) \times \tau(X) \times X \to \Z
\end{equation*}
with the following properties:
\begin{enumerate}
\item If $\deg( \mathcal{B}, W, f) \neq 0$, then $f\in \mathcal{B}(W)$.\label{enum:degree-not-zero}
\item If $\mathcal{B}\in K_1(X)$ is injective, then $\deg( \mathcal{B}, W, f) = \pm 1$ for each $f\in \mathcal{B}(W)$.
\label{enum:degree-injective}
\item If $\mathcal{A}:  X \times [0,1] \to X$ is a homotopy of compact operators on $X$ and $\mathcal{B}_s = \mathcal{I}_X + \mathcal{A}(\cdot,s)$, then for any open subset $W \subseteq X$ and $f \in X$ for which $f \not\in \mathcal{B}_s(\partial W)$ for every $s\in [0,1]$, the map $s\mapsto \deg\left( \mathcal{B}_s , W, f\right)$ is constant.
\label{enum:homotopy}
\end{enumerate}
\end{thm}

We refer to \cite{Fonseca:Book,Granas:Book} for more on Leray-Schauder degree theory, including a proof of \THM{leray-schauder-degree}.

\medskip

We want to define a homotopy between \EQ{Dirichlet-problem} and the corresponding Dirichlet problem for the Laplacian which ``stays between" the principal eigenvalues and the second eigenvalue $\lambda_2$. As a preliminary step, we must show $\maxboth < \lambda_2(F,\Omega)$.

\begin{lem}\label{lem:eigenvalues-isolated}
There exists a positive number $\eta >0$ such that $F$ has no eigenvalue $\rho$ in $\Omega$ satisfying\begin{equation*}
\maxboth < \rho \leq \maxboth + \eta.
\end{equation*} 
\end{lem}
\begin{proof}
Suppose on the contrary that there exist real numbers $\{ \lambda_k \}$ such that
\begin{equation*}
\maxboth < \lambda_k \rightarrow \maxboth
\end{equation*}
and functions $\{ u_k \} \subseteq C(\bar{\Omega})$ satisfying the Dirichlet problem
\begin{equation}\label{eq:isolated-1}
\left\{ \begin{aligned}
F(D^2u_k,Du_k,u_k,x) = {}& \lambda_k u_k & \mbox{in} & \quad \Omega \\
u_k = {}& 0 & \mbox{on} & \quad \partial \Omega,
\end{aligned}\right.
\end{equation}
and subject to the normalization $\| u_k \|_{C(\bar{\Omega})} = 1$. The $C^{1,\alpha}$ estimates imply
\begin{equation*}
\| u_k \|_{C^{1,\alpha}(\Omega)} \leq C.
\end{equation*}
By taking a subsequence, if necessary, we may assume
\begin{equation}\label{eq:eigenvalues-isolated-1}
u_k \rightarrow u \quad \mbox{uniformly on} \quad \bar{\Omega}.
\end{equation}
Passing to limits, we deduce that $u$ is a solution of the equation
\begin{equation*}
F(D^2u,Du,u,x) = \maxboth u \inOmega.
\end{equation*}
According to \THM{eigenvalues}, the function $u$ is proportional to $\varphi^+_1$ or $\varphi^-_1$. Either way, $u$ does not change sign in $\Omega$.

As in the proof of \THM{BNV}, we may choose a compact subset $K$ of $\Omega$ for which $| \Omega \backslash K |$ is small enough to ensure the operator $F - \lambda_k$ satisfies the maximum and minimum principles in $\Omega \backslash K$ for each $k\geq 1$. Recalling \EQ{eigenvalues-isolated-1}, for sufficiently large $k$, the function $u_k$ does not change sign on $K$. Then $u_k$ does not change sign in $\Omega$. Recalling \EQ{isolated-1}, we derive a contradiction to our assumption that $\lambda_k > \maxboth$.
\end{proof}

For each $0 \leq s \leq 1$, define a nonlinear operator $F_s$ by
\begin{equation}\label{eq:F-homotopy}
F_s(M,p,z,x) = -s \Gamma \trace(M) + (1-s) F(M,p,z,x).
\end{equation}
It is simple to verify that $F_s$ satisfies hypotheses \HYP{Fcontinuous}, \HYP{Felliptic}, and \HYP{Fhomogeneous}. Notice also that
\begin{equation}\label{eq:lambda-s-bounded}
\lambda^\pm_1(F_s,\Omega) \leq \lambda^\pm_1\left( \mathcal{P}^\pm(D^2\cdot) \pm \consp|D\cdot| \pm \consz | \cdot |,\Omega\right) < \infty.
\end{equation}
In order to construct a homotopy which satisfies \DEF{homotopy}, we must first verify that the functions $s\mapsto \lambda^\pm_1(F_s,\Omega)$ and $s\mapsto \lambda_2(F_s,\Omega)$ are appropriately continuous.

\begin{lem}\label{lem:lambda-1-continuous}
The maps $s \mapsto \lambda^+_1 (F_s,\Omega)$ and $s \mapsto \lambda^-_1 (F_s,\Omega)$ are continuous on $[0,1]$.
\end{lem}
\begin{proof}
Suppose $\{ s_k \} \subseteq [0,1]$ is such that $s_k \to s$. Let $\lambda_k = \lambda^-_1(F_{s_k},\Omega)$ and $\varphi_k$ denote the negative principal eigenvalue and eigenfunction of $F_{s_k}$ in $\Omega$. Recalling \EQ{lambda-leq-mu}, \EQ{eigenfunction-estimates}, and \EQ{lambda-s-bounded}, and by taking a subsequence, if necessary, we may assume there exist $\lambda \in \R$ and $\varphi\in C(\bar{\Omega})$ such that $\lambda_k \to \lambda$ and $\varphi_k \to \varphi$ uniformly on $\bar{\Omega}$. Notice $\varphi \leq 0$ and $\| \varphi \|_{C(\bar{\Omega})} = 1$. Passing to limits, we deduce the pair $(\lambda,\varphi)$ is a solution of the problem
\begin{equation*}
\left\{ \begin{aligned}
F_s(D^2\varphi,D\varphi,\varphi,x) = {}& \lambda \varphi & \mbox{in} & \quad \Omega \\
\varphi = {}& 0 & \mbox{on} & \quad \partial \Omega.
\end{aligned}\right.
\end{equation*}
According to \THM{eigenvalues} we have $\lambda = \lambda_1^-(F_s,\Omega)$. We have shown the map $s \mapsto \lambda^-_1(F_s,\Omega)$ is continuous. By arguing in a similar way, we deduce that $s \mapsto \lambda^+_1(F_s,\Omega)$ is continuous.
\end{proof}

\begin{lem}\label{lem:lambda-2-lower-semicontinuous}
The map $s \mapsto \lambda_2 (F_{s},\Omega)$ is lower semi-continuous on $[0,1]$. 
\end{lem}
\begin{proof}
Consider a sequence $s_k \to s$ for which
\begin{equation*}
\mu = \lim_{k \to \infty} \lambda_2(F_{s_k},\Omega).
\end{equation*}
We must demonstrate
\begin{equation} \label{eq:lambda-2-lower-semicontinuous-wts}
\lambda_2(F_s, \Omega) \leq \mu.
\end{equation}
If $\mu = +\infty$, then \EQ{lambda-2-lower-semicontinuous-wts} is immediate. Thus we may assume $\mu < \infty$. By taking a subsequence, if necessary, we may also assume $\lambda_2(F_{s_k},\Omega) < \infty$ for all $k$.
According to \LEM{lambda-2-is-eigenvalue}, for each $k \geq 1$ we may select an eigenfunction $\varphi_k \in C(\bar{\Omega})$ of $F_{s_k}$ in $\Omega$, with corresponding eigenvalue $\lambda_2(F_{s_k},\Omega)$, and which satisfies the normalization
\begin{equation*}
\| \varphi_k \|_{L^\infty(\Omega)} = 1.
\end{equation*}
The $C^{1,\alpha}$ estimates imply
\begin{equation*}
\sup_{k\geq 1} \| \varphi_k \|_{C^{1,\alpha}(\Omega)}  < \infty.
\end{equation*}
By taking another subsequence, we may assume $\varphi_k \to \varphi$ uniformly on $\bar{\Omega}$ for some $\varphi \in C(\bar{\Omega})$. Notice that $\| \varphi \|_{C(\bar{\Omega})} = 1$. Passing to limits, we conclude that the pair $(\varphi,\mu)$ satisfies the equation
\begin{equation}\label{eq:lower-semicontinuity-2-2}
\left\{ \begin{aligned}
F_{s}(D^2\varphi,D\varphi,\varphi,x) = {}& \mu \varphi & \mbox{in} & \quad \Omega \\
\varphi = {}& 0 & \mbox{on} & \quad \partial \Omega.
\end{aligned}\right.
\end{equation}
We now employ a familiar argument to show
\begin{equation}\label{eq:lower-semicontinuity-2-3}
\max \left\{ \lambda^-_1(F_s,\Omega),\lambda^+_1(F_s,\Omega) \right\} < \mu.
\end{equation}
The continuity of $s\mapsto \lambda_1^-(F_s,\Omega)$ implies $\lambda^-_1(F_s,\Omega) \leq \mu$. If, on the contrary $\lambda^-_1(F_s,\Omega) = \mu$, then $\varphi$ does not change sign in $\Omega$. Thus $\varphi_k$ does not change sign on a large fixed compact subset $K \subseteq \Omega$, provided $k$ is sufficiently large. If we choose $K$ large enough, then the operator
\begin{equation*}
H_k(M,p,z,x) = F(M,p,z,x) - \lambda_2(F_{s_k},\Omega)z
\end{equation*}
satisfies both the maximum and minimum principles in $\Omega\backslash K$, for all $k\geq 1$. Thus the function $\varphi_k$ does not change sign in $\Omega$, provided that $k$ is sufficiently large. This is a contradiction to \THM{eigenvalues}, verifying that $\lambda^-_1(F_s,\Omega) < \mu$. Arguing in a similar way, we obtain $\lambda^+_1(F_s,\Omega) < \mu$. We have confirmed \EQ{lower-semicontinuity-2-3}.

\medskip

Now \EQ{lambda-2-lower-semicontinuous-wts} follows from \EQ{lower-semicontinuity-2-2}, \EQ{lower-semicontinuity-2-3} and the definition \EQ{def-lambda-2} of $\lambda_2(F_s,\Omega)$.
\end{proof}

\begin{prop} \label{prop:split-the-eigenvalues}
Suppose
\begin{equation*}
\maxboth < \lambda < \lambda_2(F,\Omega).
\end{equation*}
Then there is a continuous function $\mu : [0,1] \to \R$ such that $\mu(0) = \lambda$ and
\begin{equation}\label{eq:split-the-eigenvalues}
\max\left\{ \lambda_1^-(F_s,\Omega), \lambda_1^+(F_s,\Omega) \right\} < \mu(s) < \lambda_2(F_s,\Omega) \quad \mbox{for all} \quad s\in [0,1].
\end{equation}
Moreover, for each fixed $g\in C(\Omega)$ and $p > n$, there exists a constant $C$ with the property that for any $f\in C(\Omega)\cap L^p(\Omega)$ such that $|f| \leq g$ in $\Omega$, any $0\leq s \leq 1$, and any solution $u\in C(\bar{\Omega})$ of the Dirichlet problem
\begin{equation*}
\left\{ \begin{aligned}
F_s(D^2u,Du,u,x) = {}& \mu(s) u + f & \mbox{in} & \quad \Omega \\
u = {}& 0 & \mbox{on} & \quad \partial \Omega,
\end{aligned}\right.
\end{equation*}
we have the estimate
\begin{equation}\label{eq:split-the-eigenvalues-est}
\| u \|_{C^{1,\alpha}(\Omega)} \leq C \left( 1+ \| f \|_{L^p(\Omega)}\right).
\end{equation}
\end{prop}
\begin{proof}
The existence of $\mu\in C[0,1]$ satisfying $\mu(0) = \lambda$ and \EQ{split-the-eigenvalues} follows from Lemmas \ref{lem:lambda-1-continuous} and \ref{lem:lambda-2-lower-semicontinuous}. According to the $C^{1,\alpha}$ estimates, to demonstrate \EQ{split-the-eigenvalues-est} it suffices to estimate
\begin{equation}\label{eq:split-the-eigenvalues-est-2}
\| u \|_{L^\infty(\Omega)} \leq C \left( 1 + \| f \|_{L^p(\Omega)}\right).
\end{equation}
We will establish \EQ{split-the-eigenvalues-est-2} by an indirect argument. Suppose on the contrary there are sequences $\{ f_k \}\subseteq C(\Omega)\cap L^p(\Omega)$, $\{ u_k \}\subseteq C(\bar{\Omega})$ and $\{ s_k \} \subseteq [0,1]$ satisfying $|f_k| \leq g$ in $\Omega$ as well as 
\begin{equation}
\left\{ \begin{aligned}
F_{s_k}(D^2u_k,Du_k,u_k,x) = {}& \mu(s_k) u_k + f_k & \mbox{in} & \quad \Omega \\
u_k = {}& 0 & \mbox{on} & \quad \partial \Omega,
\end{aligned}\right.
\end{equation}
but
\begin{equation*}
\| u_k \|_{L^\infty(\Omega)} / \left( 1 + \| f_k \|_{L^p(\Omega)}\right) \rightarrow \infty.
\end{equation*}
Define
\begin{equation*}
v_ k = u_k / \| u_k \|_{L^\infty(\Omega)}
\end{equation*}
and notice $v_k$ is a solution of
\begin{equation}\label{eq:split-the-eigenvalues-est-3}
F_{s_k}(D^2v_k,Dv_k,v_k,x) =  \mu(s_k) v_k + f_k/\|u_k \|_{L^\infty(\Omega)} \inOmega.
\end{equation}
Since $\| v_k \|_{L^\infty(\Omega)} = 1$, the $C^{1,\alpha}$ estimates imply
\begin{equation*}
\| v_k \|_{C^{1,\alpha}(\Omega)} \leq C \left( \| v_k \|_{L^\infty (\Omega)} + \| f_k \|_{L^p(\Omega)} / \| u_k \|_{L^\infty(\Omega)} \right) \leq C.
\end{equation*}
By taking a subsequence, if necessary, we may assume there exist $s\in [0,1]$ and $v \in C(\bar{\Omega})$ such that
\begin{equation*}
s_k \to s \quad \mbox{and} \quad v_k \to v \quad \mbox{uniformly on } \bar{\Omega}.
\end{equation*}
Moreover, on each compact subset $K \subseteq \Omega$,
\begin{equation*}
|f_k| / \| u_k \|_{L^\infty(\Omega)} \leq g / \| u_k \|_{L^\infty(\Omega)} \rightarrow 0 \quad \mbox{uniformly on} \quad K.
\end{equation*}
By passing to limits in \EQ{split-the-eigenvalues-est-3}, and noticing $\| v \|_{L^\infty(\Omega)} = 1$, we conclude that $\mu(s)$ is an eigenvalue of the operator $F_s$ in $\Omega$ with corresponding eigenfunction $v$. This contradiction to \EQ{split-the-eigenvalues} establishes \EQ{split-the-eigenvalues-est-2}.
\end{proof}

For the rest of this section, and without loss of generality, we assume that
\begin{equation}\label{eq:Fs-positive-lambda}
\mbox{the operator}\quad F_s \quad \mbox{is proper for every}\quad 0\leq s \leq 1.
\end{equation}
Otherwise, we may simply replace $F$ with the operator $G_{-\consz}$, where $\consz$ is as in \HYP{Felliptic}. We also fix a small constant $\alpha > 0$ and a function $f\in C(\Omega)\cap L^p(\Omega)$, for some $p > n$. Define a map $\mathcal{A}_{f} : C^\alpha(\Omega)\times [0,1] \to C^\alpha(\Omega)$ by 
\begin{equation*}
\mathcal{A}_{f}(v,s) = u,
\end{equation*}
where $u\in C^{1,\alpha}(\Omega)\subseteq C^\alpha(\Omega)$ is the unique solution of the Dirichlet problem
\begin{equation}
\left\{ \begin{aligned}
F_s(D^2u,Du,u,x) = {}& \mu(s) v + f & \mbox{in} & \quad \Omega \\
u = {}& 0 & \mbox{on} & \quad \partial \Omega.
\end{aligned}\right.
\end{equation}
We also define an operator $\mathcal{B}_{f,s} : C^\alpha(\Omega) \to C^\alpha(\Omega)$ by 
\begin{equation*}
\mathcal{B}_{f,s}(v) = v - \mathcal{A}_f(v,s).
\end{equation*}
For $u,v \in C^\alpha(\Omega)$, the equation
\begin{equation}\label{eq:mathcalB-equiv}
u = \mathcal{B}_{f,s}(v)
\end{equation}
is equivalent to the function $w=v-u$ solving the Dirichlet problem
\begin{equation}\label{eq:DP-mathcalB-w}
\left\{ \begin{aligned}
F_s(D^2w,Dw,w,x) = {}& \mu(s) v + f & \mbox{in} & \quad \Omega \\
w = {}& 0 & \mbox{on} & \quad \partial \Omega.
\end{aligned}\right.
\end{equation}
In this framework, our goal is to show there exists a solution $v \in C^\alpha(\Omega)$ of the equation
\begin{equation}\label{eq:existence-goal}
\mathcal{B}_{f,0}(v) = 0.
\end{equation}
We will accomplish this by demonstrating
\begin{equation} \label{eq:degree-goal}
\deg( \mathcal{B}_{f,0},W,0) \neq 0
\end{equation}
for some open subset $W \subseteq C^\alpha(\Omega)$, and then appealing to \THM{leray-schauder-degree}\ENUM{degree-injective}.

\begin{lem}\label{lem:nice-homotopy}
The map $\mathcal{A}_{f}$ is a homotopy of compact transformations on $C^\alpha(\Omega)$ in the sense of \DEF{homotopy}.
\end{lem}
\begin{proof}
For each $s \in [0,1]$, if $u = \mathcal{A}_f(v,s)$, then
\begin{equation}\label{eq:nice-homotopy-est}
\| u \|_{C^{1,\alpha}(\Omega)} \leq C \left( \max_{s\in [0,1]} |\mu(s)| \cdot \| v \|_{L^p(\Omega)} + \| f \|_{L^p(\Omega)} \right) \leq C\left( 1 + \| v \|_{L^\infty(\Omega)} \right).
\end{equation}
Thus the operator $v \mapsto \mathcal{A}_f(v,s)$ is compact for each fixed $s \in [0,1]$.

We have left to show that for each constant $C>0$, the map $(v,s) \mapsto \mathcal{A}_f(v,s)$ is uniformly continuous on the set 
\begin{equation*}
\left\{ v \in C^\alpha(\Omega) \, : \, \| v \|_{C^\alpha(\Omega)} \leq C \right\} \times [0,1].
\end{equation*}
Suppose on the contrary there exist $\varepsilon > 0$, $C>0$ and sequences $\{ s_k \}, \{ t_k \} \subseteq [0,1]$ and $\{ v_k \} \subseteq C^\alpha(\Omega)$ such that
\begin{gather*}
| s_k - t_k | \rightarrow 0, \\
\| v_k \|_{C^\alpha(\Omega)} \leq C,
\end{gather*}
but
\begin{equation}\label{eq:nice-homotopy-1}
\| u_k - \tilde{u}_k \|_{C^\alpha(\Omega)} \geq \varepsilon,
\end{equation}
where we have set $u_k = \mathcal{A}_f(v_k,s_k)$ and $\tilde{u}_k = \mathcal{A}_f(v_k,t_k)$. Recalling the estimate \EQ{nice-homotopy-est} above, and by taking a subsequence, if necessary, we can find $s\in [0,1]$ and functions $v \in C(\bar{\Omega})$ and $u,\tilde{u} \in C^\alpha(\Omega)$ such that
\begin{gather*}
s_k \rightarrow s, \\
t_k \rightarrow s, \\
v_k \rightarrow v \quad \mbox{uniformly on } \bar{\Omega},\\
u_k \rightarrow u \quad \mbox{in } C^\alpha(\Omega),\\
\intertext{and}
\tilde{u}_k \rightarrow \tilde{u} \quad \mbox{in } C^\alpha(\Omega).
\end{gather*}
Passing to limits, we deduce that $u$ and $\tilde{u}$ are both solutions of the problem
\begin{equation*}
\left\{ \begin{aligned}
F_s(D^2u,Du,u,x) = {}& \mu(s) v + f & \mbox{in} & \quad \Omega \\
u = {}& 0 & \mbox{on} & \quad \partial \Omega.
\end{aligned}\right.
\end{equation*}
Recalling \EQ{Fs-positive-lambda}, we conclude that $u = \tilde{u}$, which contradicts \EQ{nice-homotopy-1}. This completes the proof.
\end{proof}

\begin{lem}\label{lem:degree-B-1}
Let $R = 1 + C\left( 1 + \| f \|_{L^p(\Omega)} \right)$, where $C$ is the constant in \EQ{split-the-eigenvalues-est}, and let $W\subseteq C^\alpha(\Omega)$ denote the ball
\begin{equation}\label{eq:degreeW}
W = \left\{ v \in C^\alpha(\Omega) \, : \, \| v \|_{C^\alpha(\Omega)} < R \right\}.
\end{equation}
Then
\begin{equation}\label{eq:degree-at-1}
\deg ( \mathcal{B}_{f,1} , W , 0) = \pm 1.
\end{equation}
\end{lem}
\begin{proof}
We will show $\mathcal{B}_{f,1}$ is injective and then apply \THM{leray-schauder-degree}\ENUM{degree-injective}. As mentioned above, using the definitions, it is easy to see the equation $u = \mathcal{B}_{f,1}(v)$ is equivalent to the function $w = v - u$ satisfying the Dirichlet problem
\begin{equation*}
\left\{ \begin{aligned}
-\Gamma \Delta w = {}& \mu(1) v + f & \mbox{in} & \quad \Omega \\
w = {}& 0 & \mbox{on} & \quad \partial \Omega,
\end{aligned}\right.
\end{equation*}
which can be rewritten as
\begin{equation}\label{eq:DP-w-2}
\left\{ \begin{aligned}
-\Gamma \Delta w - \mu(1)w= {}& \mu(1) u + f & \mbox{in} & \quad \Omega \\
w = {}& 0 & \mbox{on} & \quad \partial \Omega.
\end{aligned}\right.
\end{equation}
Since $\lambda_1(-\Gamma \Delta, \Omega) < \mu(1) < \lambda_2(- \Gamma \Delta, \Omega)$, the Fredholm theory for linear compact operators implies \EQ{DP-w-2} has a unique solution $w\in W^{2,p}(\Omega)$ for each $u\in C^\alpha(\Omega)$. Hence the equation
\begin{equation*}
\mathcal{B}_{f,1}(v) = u
\end{equation*}
has a unique solution $v\in C^\alpha(\Omega)$ for each $u \in C^\alpha(\Omega)$. That is, $\mathcal{B}_{f,1}$ is bijective. According to estimate \EQ{split-the-eigenvalues-est}, the unique solution $v\in C^\alpha(\Omega)$ of
\begin{equation*}
\mathcal{B}_{f,1}(v) = 0
\end{equation*}
must satisfy the estimate $\| v \|_{C^\alpha(\Omega)} < R$, and hence $v \in W$. Now \EQ{degree-at-1} follows from an application of \THM{leray-schauder-degree}\ENUM{degree-injective}.
\end{proof}

We are now in a position to confirm \EQ{degree-goal}.

\begin{prop} \label{prop:degree-B-0}
Let $W$ be as in \LEM{degree-B-1}. Then
\begin{equation}\label{eq:degree-at-0}
\deg ( \mathcal{B}_{f,0} , W , 0) = \pm 1.
\end{equation}
\end{prop}
\begin{proof}
According to \EQ{split-the-eigenvalues-est}, we have $0 \not\in B_{f,s}(\partial W )$ for every $s\in [0,1]$. Now apply \LEM{nice-homotopy} and \THM{leray-schauder-degree}\ENUM{homotopy} to deduce
\begin{equation*}
\mbox{the map} \quad s \mapsto \deg \left( \mathcal{B}_{f,s}, W, 0 \right) \quad \mbox{is constant}.
\end{equation*}
Recalling \EQ{degree-at-1}, we conclude
\begin{equation*}
\deg \left( \mathcal{B}_{f,s}, W, 0 \right) = \deg \left( \mathcal{B}_{f,1}, W, 0 \right) = \pm 1
\end{equation*}
for every $s\in [0,1]$. In particular, $\deg ( \mathcal{B}_{f,0} , W , 0) = \pm 1$.
\end{proof}

We will now complete the proof of \THM{existence-past-lambda-1}.

\begin{proof}[Proof of \THM{existence-past-lambda-1}]
According to \THM{leray-schauder-degree}\ENUM{degree-not-zero} and \PROP{degree-B-0}, there exists a function $u \in W \subseteq C^\alpha(\Omega)$ such that $\mathcal{B}_{f,0}(u) = 0$. Recalling the definition of $\mathcal{B}_{f,s}$, we see immediately that $u$ is a solution of the Dirichlet problem
\begin{equation}
\left\{ \begin{aligned}
F(D^2u,Du,u,x) = {}& \mu(0) u + f  & \mbox{in} & \quad \Omega \\
u = {}& 0 & \mbox{on} & \quad \partial \Omega. 
\end{aligned}\right.
\end{equation}
According to the $C^{1,\alpha}$ estimates, $u\in C^{1,\alpha}(\Omega)$. Recalling $\mu(0) = \lambda$, the proof is complete.
\end{proof}

\begin{rem}
I do not know whether the solutions of \EQ{Dirichlet-problem} given by \THM{existence-past-lambda-1} are unique, in general, even in the case that $F$ is convex in $(M,p,z)$.
\end{rem}



\section{An anti-maximum principle} \label{sec:anti-maximumprinciple}

The purpose of this section is to establish \THM{AMP}. Our argument, similar to the method employed in 
\cite{Birindelli:1995}, is based on Hopf's Lemma and the following two nonexistence results.

\begin{prop}\label{prop:PDEnonexistence}
Assume $\lambda \geq \lambda_1^+(F,\Omega)$ and $f\in C(\Omega)$ is such that $f\geq 0$ and $f \not\equiv 0$. Then the problem
\begin{equation}\label{eq:PDEnonexistence}
\left\{ \begin{aligned}
F(D^2u,Du,u,x) \geq {}& \lambda u + f & \mbox{in} & \quad \Omega \\
u \geq {}& 0 & \mbox{on} & \quad \partial \Omega
\end{aligned}\right.
\end{equation}
has no nonnegative solution $u\in C(\bar{\Omega})$. Moreover, under the additional assumption
\begin{equation*}
\lambda_1^+(F,\Omega) \leq \lambda \leq \lambda_1^-(F,\Omega),
\end{equation*}
problem \EQ{PDEnonexistence} does not possess a solution $u\in C(\bar{\Omega})$.
\end{prop}
\begin{proof}
Suppose a solution $u\geq 0$ of \EQ{PDEnonexistence} exists under the assumptions that $\lambda \geq \lambda_1^+(F,\Omega)$ and $f\in C(\Omega)$ is such that $f\geq 0$. We claim $f\equiv 0$. If $u \equiv 0$ we have nothing to show, so suppose $u\not\equiv 0$. According to Hopf's Lemma, $u > 0$. The definition of $\lambda^+_1(F,\Omega)$ implies $\lambda \leq \lambda^+_1(F,\Omega)$, from which we deduce $\lambda=\lambda^+_1(F,\Omega)$. Applying \THM{BNV}, we see that $u$ is a positive constant multiple of the eigenfunction $\varphi^+_1$. This implies $f \equiv 0$, as desired.

Now suppose in addition that $\lambda \leq \lambda^-_1(F,\Omega)$, and that there exists a solution $u$ of \EQ{PDEnonexistence}. According to our argument above, if $f\not \equiv 0$, then $u$ must be negative somewhere in $\Omega$. Applying \THM{BNV}, we see that $u\equiv t\varphi^-_1$ for some $t > 0$. In particular, this implies $\lambda = \lambda^-_1(F,\Omega)$ and $f\equiv 0$, a contradiction.
\end{proof}

Arguing in a similar fashion, we obtain the following proposition.

\begin{prop}\label{prop:PDEnonexistence-2}
Assume $\lambda \geq \lambda_1^-(F,\Omega)$ and $f\in C(\Omega)$ is such that $f\leq 0$ and $f\not\equiv 0$. Then the problem
\begin{equation}\label{eq:PDEnonexistence-2}
\left\{ \begin{aligned}
F(D^2u,Du,u,x) \leq {}& \lambda u + f & \mbox{in} & \quad \Omega \\
u \leq {}& 0 & \mbox{on} & \quad \partial \Omega
\end{aligned}\right.
\end{equation}
has no nonpositive solution $u \in C(\bar{\Omega})$. Moreover, under the additional assumption 
\begin{equation}\label{eq:equal-eigenvalues}
\lambda_1^-(F,\Omega) \leq \lambda \leq  \lambda_1^+(F,\Omega),
\end{equation}
problem \EQ{PDEnonexistence-2} does not possess a solution $u \in C(\bar{\Omega})$.
\end{prop}

We are now ready for the proof of \THM{AMP}.

\begin{proof}[Proof of of \THM{AMP}]
We will prove only \ENUM{AMP1}, since the argument for \ENUM{AMP2} is similar. Assume on the contrary there are sequences $\{ \lambda_k\} \subseteq \left( \lambda^-_1(F,\Omega), \infty\right)$ and $\{ u_k \}\subseteq C(\bar{\Omega})$ such that
\begin{equation*}
\lambda_k  \to \lambda_1^-(F,\Omega)
\end{equation*}
and $u_k$ satisfies
\begin{equation}\label{eq:AMB-false-k}
\left\{ \begin{aligned}
F(D^2u_k,Du_k,u_k,x) = {}& \lambda_k u_k + f & \mbox{in} & \quad \Omega \\
u_k = {}& 0 & \mbox{on} & \quad \partial \Omega;
\end{aligned}\right.
\end{equation}
but each $u_k$ is nonnegative somewhere in $\Omega$. Select $x_k\in \Omega$ such that $u_k$ attains a nonnegative local maximum at $x_k$. In particular, we have
\begin{equation}\label{eq:sequenceantimax1}
u_k(x_k) \geq 0 \quad \mbox{and} \quad Du_k(x_k) = 0.
\end{equation}
By taking a subsequence, if necessary, we may assume
\begin{equation}\label{eq:sequenceantimax2}
x_k \rightarrow x_0
\end{equation}
for some $x_0\in \bar{\Omega}$. We now assert
\begin{equation}\label{eq:unboundeduk}
\sup_{k\geq 1} \| u_k \|_{L^\infty(\Omega)} = \infty.
\end{equation}
On the contrary, suppose
\begin{equation*}
\sup_{k\geq 1} \| u_k \|_{L^\infty(\Omega)} < \infty.
\end{equation*}
By once again applying the $C^{1,\alpha}$ estimates and taking a subsequence, we find a function $u \in C^{1,\alpha}(\Omega)$ such that $u_k \rightarrow u$ uniformly on $\bar{\Omega}$. Passing to limits in \EQ{AMB-false-k}, we see that $u$ is a solution of the problem
\begin{equation*}
\left\{ \begin{aligned}
F(D^2u,Du,u,x) = {}& \lambda_1^-(F,\Omega) u + f & \mbox{in} & \quad \Omega \\
u = {}& 0 & \mbox{on} & \quad \partial \Omega,
\end{aligned}\right.
\end{equation*}
in violation of \PROP{PDEnonexistence}. We have demonstrated \EQ{unboundeduk}.

\medskip

Without loss of generality, we may assume $\| u_k \|_{L^\infty(\Omega)} \rightarrow \infty$. Normalize by setting
\begin{equation*}
v_k = u_k/\| u_k \|_{L^\infty(\Omega)},
\end{equation*}
and verify that $v_k$ satisfies
\begin{equation}\label{eq:AMPvk}
\left\{ \begin{aligned}
F(D^2v_k,Dv_k,v_k,x) = {}& \lambda_k v_k + f/\|u_k\|_{L^\infty(\Omega)} & \mbox{in} & \quad \Omega \\
v_k = {}& 0 & \mbox{on} & \quad \partial \Omega.
\end{aligned}\right.
\end{equation}
By using the $C^{1,\alpha}$ estimates again, and taking a subsequence, we may assume 
\begin{equation}\label{eq:AMPvktov}
v_k \rightarrow v \quad \mbox{in} \quad C^1(\bar{\Omega})
\end{equation}
for some $v\in C^{1,\alpha}(\Omega)$. Passing to limits in \EQ{AMPvk}, we deduce that $v$ is a solution of
\begin{equation*}
\left\{ \begin{aligned}
F(D^2v,Dv,v,x) = {}& \lambda_1^-(F,\Omega) v & \mbox{in} & \quad \Omega \\
v = {}& 0 & \mbox{on} & \quad \partial \Omega.
\end{aligned}\right.
\end{equation*}
Recalling $\| v \|_{L^\infty(\Omega)} =1$, we have $v\equiv t\varphi^-_1$ for $t=\pm 1$. We will finish the proof by demonstrating that both $t=1$ and $t=-1$ lead to contradictions.

If $t =-1$ then $v \equiv - \varphi^-_1 > 0$ in $\Omega$. Choose a compact set $K\subseteq \Omega$ large enough that the maximum principle holds for the operator $F - \lambda_k$ in $\Omega\backslash K$, for each $k \geq 1$. Recalling \EQ{AMPvktov}, if we take $k$ to be sufficiently large, then $v_k > 0$ on $K$. Thus $v_k \geq 0$ on $\Omega$ due to the assumption $f \geq 0$ and the maximum principle. Recalling \EQ{AMPvk}, we derive a contradiction to \PROP{PDEnonexistence}. Thus $t = -1$ is impossible.

Finally, consider the case $t = 1$, which implies $v < 0$ in $\Omega$. Using \EQ{sequenceantimax1} and \EQ{sequenceantimax2}, we deduce $v(x_0) = 0$, and hence $x_0\in \partial \Omega$. However, we also deduce $Dv(x_0) = 0$, in violation of Hopf's Lemma. This rules out the possibility $t = 1$, which completes the proof of the theorem.
\end{proof}

We now state a more general anti-maximum principle. Consider a family $\left\{ F_s \, : \, s \in [0,1] \right\}$ of nonlinear operators such that for each fixed $s \in [0,1]$, the nonlinearity $F_s$ satisfies \HYP{Fcontinuous}, \HYP{Felliptic}, and \HYP{Fhomogeneous}. We also assume that the map
\begin{equation*}
(M,p,z,x,s) \mapsto F_s(M,p,z,x)
\end{equation*}
is continuous on $\Sy \times \R^n \times \R \times \Omega \times [0,1]$. An examination of the proof of \THM{AMP} convinces us the same argument may be employed, with only minor modifications, to establish the following result.

\begin{thm}
Assume $p > n$ and $f\in C(\Omega\times [0,1])$ is such that $f\geq 0$, $f(\cdot,s) \in L^p(\Omega)$ and $f(\cdot,s) \not\equiv 0$ for every $s\in [0,1]$. Suppose
\begin{equation*}
\lambda^+_1(F_s,\Omega) \leq \lambda^-_1(F_s,\Omega) < \lambda^-_1(F_0,\Omega) \quad \mbox{for all} \quad 0 < s \leq 1.
\end{equation*}
Then there exists a constant $\eta = \eta(f) > 0$ such that for every $0 < s < \eta$, any solution $u\in C(\bar{\Omega})$ of the Dirichlet problem
\begin{equation*}
\left\{ \begin{aligned}
F_s(D^2u,Du,u,x) = {}& \lambda^-_1(F_0,\Omega) u + f(\cdot,s) & \mbox{in} & \quad \Omega \\
u = {}& 0 & \mbox{on} & \quad \partial \Omega
\end{aligned}\right.
\end{equation*}
must satisfy $u < 0$ in $\Omega$.
\end{thm}

We leave the statement of the corresponding generalization of \THM{AMP}\ENUM{AMP2} to the reader.


\appendix

\section{Hopf's Lemma} \label{app:hopf}

Because we could not find a short and simple proof in the literature, we present a complete proof of Hopf's Lemma for viscosity solutions. 

\begin{thm}[Hopf's Lemma]\label{thm:hopf}
Assume the domain $\Omega$ satisfies an interior sphere condition. Suppose $u\in C(\bar{\Omega})$ is such that $u\not\equiv 0$ and $u\leq 0$ in $\Omega$, and suppose that $u$ is a viscosity subsolution of
\begin{equation}\label{eq:hopfeq}
\pucci(D^2u) - \consp |Du| - \consz |u| \leq 0 \inOmega.
\end{equation}
Then for each $x_0\in \partial \Omega$ such that $u(x_0)=0$, we have
\begin{equation}
\limsup_{h\to 0+} \frac{u(x_0) - u(x_0+h \xi)}{h} > 0
\end{equation}
for every $\xi \in \R^n$ such that $\xi \cdot \nu < 0$, where $\nu$ is the outer unit normal vector to any sphere which touches $\partial \Omega$ from the interior at $x_0$.
Moreover, $u < 0$ in $\Omega$.
\end{thm}
\begin{proof}
Suppose first that $u < 0$ in $\Omega$. Select a ball $B(y,R) \subseteq \Omega$ such that $B \cap \partial\Omega = \{ x_0 \}$. Let $V$ be the annular region
\begin{equation*}
V = \left\{ x\in \R^n : R/2 < |x-y| < R \right\}.
\end{equation*}
Assume without loss of generality $y=0$, and define a function $v$ by
\begin{equation*}
v(x) = e^{-\alpha R^2} - e^{-\alpha |x|^2}
\end{equation*}
where $\alpha > 0$ will be selected below. Notice
\begin{equation*}
-e^{-\alpha |x|^2} < v(x) < 0
\end{equation*}
for $x\in V$. We claim that if $\alpha > 0$ is large enough, then $v$ is a supersolution of
\begin{equation}\label{eq:hopfvsuper}
\pucci(D^2v ) - \consp |Dv| - \consz |v| > 0 \quad \mbox{in} \quad \bar{V}.
\end{equation}
For $x\in \bar{V}$, we estimate
\begin{align*}
\lefteqn{\pucci(D^2v(x) ) - \consp |Dv(x)| - \consz |v(x)|} \qquad \qquad & \\ 
& \geq 2 \alpha e^{-\alpha |x|^2 } \pucci ( I - 2\alpha x \otimes x ) - 2\alpha \consp |x| e^{-\alpha |x|^2} - \consz |v(x)| \\
& \geq e^{-\alpha |x|^2} \left( -2\alpha n \Gamma + 4 \alpha^2 \gamma |x|^2 - 2 \alpha \consp |x| - \consz \right) \\
& \geq e^{-\alpha |x|^2} \left( -2\alpha n \Gamma + \alpha ^2 \gamma R^2 - 2 \alpha \consp R - \consz \right ).
\end{align*}
Thus $v$ satisfies \EQ{hopfvsuper} provided we select $\alpha > 0$ large enough that $\alpha \geq 1$ and 
\begin{equation*}
\alpha > \frac{2n\Gamma + 2\consp R + \consz}{R^2 \gamma}. 
\end{equation*}
Since $u < 0$ in $\Omega$, we may select $\varepsilon > 0$ small enough that
\begin{equation*}
u(x) \leq \varepsilon v(x) \quad \mbox{for all } |x| = R/2.
\end{equation*}
We also have
\begin{equation*}
u(x) \leq 0 = \varepsilon v(x) \quad \mbox{for all } |x| = R.
\end{equation*}
Therefore, we may apply the comparison principle to deduce that
\begin{equation*}
u \leq \varepsilon v \quad \mbox{in} \quad V.
\end{equation*}
Hence
\begin{eqnarray*}
\limsup_{h\to 0+} \frac{- u(x_0+h \xi)}{h} & \geq & \varepsilon \limsup_{h\to 0+}\frac{- v(x_0+h \xi)}{h} \\
& = & \varepsilon \limsup_{h\to 0+} \frac{-e^{-\alpha R^2} + e^{-\alpha |x_0+h\xi|^2}}{h} \\ 
& \geq & - 2\alpha \varepsilon e^{-\alpha R^2} x_0\cdot \xi \\
& > & 0.
\end{eqnarray*}
This completes our proof in the case that $u < 0$ in $\Omega$.

Now suppose that $u = 0$ somewhere in $\Omega$. Let $W\subseteq \Omega$ be the zero set of $u$ in $\Omega$. Select $y\in \Omega$ such that $u(y) < 0$ and the distance from $y$ to $W$ is less than the distance from $y$ to $\partial\Omega$. Find $x_0\in W$ such that $|x_0-y|=R$, where $R = \inf_{x\in W} |x-y|$. Once again, assume $y=0$ and let $V$ and $v$ be defined as above. Following the same procedure as above, we conclude that $u\leq \varepsilon v$ in $V$, as well as $u \leq 0 \leq \varepsilon v$ for $|x|\geq R$. Since $u(x_0) = 0 = \varepsilon v(x_0)$, it follows that
\begin{equation*}
x \mapsto u(x) - \varepsilon v(x) \quad \mbox{has a local maximum at} \quad x=x_0.
\end{equation*}
Since $u$ is a viscosity subsolution of \EQ{hopfeq}, we have
\begin{equation*}
\varepsilon \left[ F(D^2v(x_0),Dv(x_0),v(x_0), x_0) - \consp|Dv(x_0)| \right] \leq 0.
\end{equation*}
This contradicts \EQ{hopfvsuper}, completing the proof.
\end{proof}

\bibliographystyle{elsart-num-sort}
\bibliography{bibdb1}

\end{document}